\documentclass[11pt]{amsart}
\usepackage{euscript}
\usepackage{amssymb}
\usepackage{amsmath}
\usepackage{epsfig}
\usepackage{epic}
\usepackage[matrix,arrow,curve]{xy}

\newtheoremstyle{my}{1.5em}{0.5em}{\em}{}{\sc}{.}{0.5em}{}
\theoremstyle{my}
\newtheorem{thm}{Theorem}
\newtheorem{Theorem}[thm]{Theorem}
\newtheorem*{Theorem*}{Theorem}
\newtheorem{Corollary}[thm]{Corollary}
\newtheorem{corollary}[thm]{Corollary}
\newtheorem*{corollary*}{Corollary}
\newtheorem{Lemma}[thm]{Lemma}
\newtheorem{lemma}[thm]{Lemma}

\newtheorem{conjecture}[thm]{Conjecture}

\newtheorem{definition}[thm]{Definition}

\newtheorem{remark}[thm]{Remark}

\newtheorem{example}[thm]{Example}

\newtheorem*{examples*}{Examples}

\newcommand{\Acknowledgments}{{\em Acknowledgments.} }
\newcommand{\ignore}[1]{...}
\newcommand{\R}{\mathbb{R}}
\newcommand{\Z}{\mathbb{Z}}
\newcommand{\Q}{\mathbb{Q}}
\newcommand{\C}{\mathbb{C}}

\newcommand{\iso}{\cong}           %isomorphism sign
\newcommand{\htp}{\simeq}          %homotopy sign
\newcommand{\smooth}{C^\infty}

\newcommand{\RP}[1]{\R {\mathrm P}^{#1}}

\newcommand{\anti}{\mathrm{anti}}
\newcommand{\inv}{\mathrm{inv}}

\newcommand{\Sym}{\mathrm{Sym}}

\newcommand{\Jac}{\mathrm{Jac}}
\newcommand{\id}{\mathrm{id}}

\newcommand{\cok}{\mathrm{coker}}

\newcommand{\Conf}{\mathrm{Conf}}
\newcommand{\Kh}{\mathit{Kh}}
\newcommand{\Hilb}{\mathrm{Hilb}}

\newcommand{\Slice}{\mathcal{S}}

\renewcommand{\o}{\omega}

\renewcommand{\SS}{\mathcal{S}}

\newcommand{\Y}{\EuScript{Y}}

\newcommand{\scrH}{\EuScript{H}}
\newcommand{\scrP}{\EuScript{P}}

\newcommand{\scrM}{\EuScript{M}}
\newcommand{\scrS}{\EuScript{S}}

\newcommand{\CC}{\EuScript{C}}
\newcommand{\DD}{\mathbf{D}}
\newcommand{\XX}{\mathbf{X}}

\newcommand{\diag}{\mathrm{diag}}

\renewcommand{\SS}{\mathbf{S}}

\newcommand{\TT}{\EuScript{T}}
\newcommand{\HH}{\EuScript{H}}

\newcommand{\X}{\EuScript{X}}
\begin{document}
\title[Localization for involutions]{Localization for involutions \\ in Floer cohomology}
\author{Paul Seidel, \, Ivan Smith}
\date{v3; August 2010}
\begin{abstract}
We consider Lagrangian Floer cohomology for a pair of Lagrangian submanifolds in a symplectic manifold $M$. Suppose that $M$ carries a symplectic involution, which preserves both submanifolds. Under various topological hypotheses, we prove a localization theorem for Floer cohomology, which implies a Smith-type inequality for the Floer cohomology groups in $M$ and its fixed point set. Two applications to symplectic Khovanov cohomology are included.
\end{abstract}
\maketitle

\section{Introduction}
Suppose that we have a compact smooth manifold $M$ with a smooth action of the group $G = \Z/2$ on it, whose fixed point set we denote by $M^{\mathrm{inv}}$. The classical Smith inequality \cite{Smith-inequality, Bredon} asserts that
\begin{equation} \label{eq:old-smith}
\mathrm{dim}\, H^*(M;\Z_2) \geq \mathrm{dim}\, H^*(M^{\mathrm{inv}};\Z_2).
\end{equation}
This is an inequality of total dimensions, and does not hold separately in each degree (similar inequalities hold for actions of $G = \Z/p$ with prime $p$; such generalizations will not be considered in this paper). There are several ways of proving \eqref{eq:old-smith}. One of them involves equivariant cohomology $H^*_G(M;\Z_2)$ and the localization theorem, which says that the restriction map
\begin{equation} \label{eq:restriction-map}
H^*_G(M;\Z_2) \longrightarrow H^*(M^{\mathrm{inv}};\Z_2) \otimes_{\Z_2} H^*(BG;\Z_2)
\end{equation}
becomes an isomorphism once one inverts the generator $q$ of $H^*(BG;\Z_2) = \Z_2[[q]]$. The purpose of this paper is develop an analogue of this for Lagrangian Floer cohomology, with a view to applications to the link invariant introduced in \cite{SS}.

The overall setup is as follows (see Section \ref{subsec:borel-floer} for more details). We consider a symplectic manifold $M$, which is assumed to be exact and convex at infinity, and carry a symplectic involution $\iota$. $L_0,L_1 \subset M$ are exact Lagrangian submanifolds invariant under $\iota$, which are supposed to be either compact or else reasonably well-behaved at infinity. We denote the fixed point set by $M^{\mathrm{inv}}$. The fixed parts $L_k^{\mathrm{inv}} = L_k \cap M^{\mathrm{inv}}$ are again automatically Lagrangian submanifolds. There are three kinds of Floer cohomology groups involved: the ordinary Floer cohomology $HF(L_0^{\mathrm{inv}},L_1^{\mathrm{inv}})$ in $M^{\mathrm{inv}}$, the corresponding group $HF(L_0,L_1)$ in $M$, and the $G$-equivariant analogue of the latter. The first two are taken with $\Z_2$ coefficients, and the last one is a module over $\Z_2[[q]]$.

The exactness assumptions put us in a situation very close to standard Morse theory. In spite of that, the analogue of the Smith inequality fails even in relatively simple examples (see Section \ref{subsec:moebius}), due to phenomena which have no analogues in the classical context. In order to rule them out, we have to impose another strong restriction on the local topology near the fixed point set, namely the existence of a {\em stable normal trivialization} (defined in Section \ref{subsec:stable-normal}).

\begin{Theorem}\label{thm:mainsmith}
If a stable normal trivialization exists, the Smith inequality holds:
\begin{equation}
\mathrm{dim}\, HF(L_0,L_1) \geq \mathrm{dim}\, HF(L_0^{\mathrm{inv}}, L_1^{\mathrm{inv}}).
\end{equation}
\end{Theorem}

As in the classical setting, we will deduce this from the existence of a localization map from equivariant Floer cohomology to Floer cohomology in the fixed point set (Theorem \ref{th:th}). A good finite-dimensional model is equivariant Morse theory, in the context where the Morse function does not have positive definite Hessian in normal direction to the fixed point set. In this case, \eqref{eq:restriction-map} is not just given by projection to a quotient complex, and instead involves higher-dimensional moduli spaces. The way we extract information from these spaces was influenced by Donaldson's account in \cite{Donaldson-Floerbook}. The whole of Section \ref{sec:morse} is devoted to a detailed explanation of the Morse-theoretic construction, where we are careful to use only arguments which are readily transplanted into Floer theory. In particular, we never appeal to the isomorphism between Morse and singular cohomology.

The existence of stable normal trivializations allows us to easily achieve equivariant transversality, which we rely on to define the localization map (in contrast, the definition of equivariant Floer cohomology requires no such assumption). As pointed out above, this is not just a technical condition. Our best guess for the more general situation is that the localization map should take values in a version of $HF^*(L_0,L_1)$ twisted by normal contributions (see Remark \ref{th:twist}). This idea reflects the influence of a somewhat different approach to localization in monopole Floer homology, taken by Kronheimer and Mrowka in \cite{KM}.

In spite of the restrictions on its validity, the Floer-theoretic Smith inequality has nontrivial applications. We include two of them, both concerning the invariant $\Kh_{\mathrm{symp}}^*(\kappa)$ of oriented links $\kappa \subset S^3$. This was introduced in \cite{SS}, and conjectured to be isomorphic to Khovanov's combinatorial construction \cite{khovanov-jones}, but is quite hard to compute from the definition. First, we relate the symplectic Khovanov cohomology of $\kappa$ to the Heegaard Floer homology of the double branched cover $N_\kappa \rightarrow S^3$ via spectral sequence arguments. This is similar to the relation between combinatorial Khovanov cohomology and Heegaard Floer theory established in \cite{OS-doublecover}, even though the actual constructions are quite different. In particular, we get:

\begin{corollary} \label{cor:mainone}
For any oriented link $\kappa \subset S^3$,
\begin{equation}
\mathrm{dim}\, \Kh^*_{\mathrm{symp}}(\kappa;\Z_2) \geq 2\cdot |H^2(N_\kappa;\Z)|.
\end{equation}
\end{corollary}

The right hand side (which, by definition, vanishes if $b_1(N_\kappa) > 0$) is also (twice) the value of the Alexander polynomial at $-1$, known classically as the determinant of the knot \cite[Corollary 9.2]{Lickorish}.  There is (currently unpublished) a spectral sequence from combinatorial to symplectic Khovanov homology, at least with $\Z_2$ coefficients. Together with the results of Section \ref{sec:four} below, and the main theorem of \cite{OS-doublecover}, this implies that for alternating knots the combinatorial and symplectic Khovanov homologies with $\Z_2$ coefficients have the same rank. The localization map also provides additional information, which has no direct counterpart in any other known construction. This should give rise to new link invariants, called cokernel polynomials, conjecturally of relevance to the Jones polynomial (see Section \ref{subsec:speculate} for a discussion of this and other possible further developments).

For our second application, consider the double covering $S^3 \rightarrow S^3$ branched over the unknot. Let $\bar\kappa \subset S^3$ be an oriented link which avoids the branch locus, and $\kappa$ its preimage.

\begin{corollary} \label{cor:maintwo}
For any $\kappa$ as described above,
\begin{equation}
\mathrm{dim}\,\Kh_{\mathrm{symp}}^*(\kappa;\Z_2) \geq \mathrm{dim}\, \Kh^*_{\mathrm{symp}}(\bar\kappa;\Z_2).
\end{equation}
\end{corollary}

In the original combinatorial Khovanov theory, any generator of the chain complex for $\bar\kappa$ lifts to two generators of the complex for $\kappa$, but their bi-degrees are in general unrelated. As far as we know, the analogue of Corollary \ref{cor:maintwo} for the combinatorial theory is not known to hold.

\Acknowledgments The first author would like to thank Tom Mrowka and Tim Perutz for explaining aspects of monopole Floer homology to him. The second author would like to thank Robert Lipshitz, Ciprian Manolescu, Peter Ozsv\'ath, and Jake Rasmussen for helpful discussions on the material of Section \ref{sec:four}. The first author would also like to apologize to the second author for the long delay (more than five years) in getting this paper ready for publication.

The first author was partially supported by NSF grants DMS-0405516 and DMS-0652620. The second author was partially supported by European Research Council grant ERC-2007-StG-205349.

\section{Morse theory\label{sec:morse}}

\subsection{Equivariant Morse cohomology\label{sec:equivariant-morse}}

Let $M$ be a closed smooth manifold, equipped with an action of the group $G = \Z/2$, or equivalently with an involution $\iota$. Form the Borel construction
\begin{equation} \label{eq:borel-constr}
M_{\mathrm{borel}} = M \times_G EG = M \times_G S^{\infty} \longrightarrow \RP{\infty} = BG = S^\infty/G.
\end{equation}
This is a locally trivial bundle with fibre $M$ and monodromy $\iota$ around the non-trivial element of $\pi_1(\RP{\infty}) = G$. Classically, one defines equivariant cohomology as the cohomology of $M_{\mathrm{borel}}$, and this approach is also suitable in a Morse-theoretic context, in particular if one wishes to avoid equivariant transversality issues. We will explain one version of the resulting definition, chosen for its easy extendability to Floer theory. The particular advantage is that the underlying moduli spaces are solutions of a gradient-type equation on $M$, with the $\RP{\infty}$ direction serving only as a parameter space. This is the approach used to define general family Floer homology theories in \cite{hutchings}; a related idea for $S^1$-equivariant Floer homology appears earlier in \cite{viterbo-functorsI}.

Equip $\RP{\infty}$ with its standard Morse function $h$, whose pullback to the covering space $S^{\infty} \subset \R^{\infty}$ is $h(z_0,z_1,z_2,\dots) = |z_1|^2 + 2|z_2|^2 + \cdots$, and with its standard round Riemannian metric (technically, it may be better to think of the infinite-dimensional projective space as the union of its finite-dimensional counterparts, equipped with the restrictions of the metric and Morse function). $h$ has exactly one critical point $z^{(k)}$ of each index $k \geq 0$, and the associated gradient flow is regular (Morse-Smale). Next, choose a smooth family of functions on the fibres of \eqref{eq:borel-constr}, or equivalently a family of functions $f_z$ on $M$ parametrized by $z \in S^\infty$ and satisfying $f_{-z} = f_z \circ \iota$. We also want a corresponding family of metrics (to keep the notation reasonably short, we don't give names to the metrics; nevertheless, whenever we talk about $\nabla f_z$, it is implicit that the associated metric should be used). This choice should be made in such a way that the functions $f^{(k)} = f_{z^{(k)}}$ associated to the critical points are Morse. Moreover, for technical convenience, we assume that our family is locally constant in a small neighbourhood of each $z^{(k)}$. Denote by $C^{*,(k)}$ the Morse cochain space generated by the critical points of $f^{(k)}$. Define the Borel-type equivariant cochain group to be
\begin{equation} \label{eq:borel-complex}
C^*_{\mathrm{borel}} = \prod_{k=0}^\infty C^{*-k,(k)}.
\end{equation}
To see how the differential on this is constructed, take a map $v: \R \rightarrow \RP{\infty}$ which is a gradient flow line of $h$, with limits $z^{(j)}$ and $z^{(k)}$ at negative respectively positive times. Take also a map $u: \R \rightarrow M_{\mathrm{borel}}$ which lifts $v$, and which satisfies the equation
\begin{equation} \label{eq:u-equation}
(du/ds)^{\mathrm{vert}} = \nabla f_{v(s)}(u(s)).
\end{equation}
Here, we have used the locally trivial connection on $M_{\mathrm{borel}} \rightarrow \RP{\infty}$ to take the vertical, or fibrewise, component of $du/ds$. If $j = k$, meaning that $v$ is constant, this is just a solution of the gradient flow equation for $f^{(j)}$, so its limits will be critical points $x_{\pm}^{(j)}$ of $f^{(j)}$. In the other cases, one can simplify the equation by trivializing $M_{\mathrm{borel}} \rightarrow \RP{\infty}$ along (the closure of) $v$. With respect to this trivialization, $f_{v(s)}$ is a family $f_v$ of functions on $M$, depending on the parameter $s \in \R$, with the property that $f_{v,s} = f^{(j)}$ for $s \ll 0$, $f_{v,s} = f^{(k)}$ for $s \gg 0$. In these terms, \eqref{eq:u-equation} is the associated continuation map equation on $M$,
\begin{equation}
du/ds = \nabla f_{v,s}(u(s)).
\end{equation}
Clearly, the limits as $s \rightarrow \pm\infty$ will be critical points $x_-^{(j)}$ of $f^{(j)}$ and $x_+^{(k)}$ of $f^{(k)}$. In either situation, fixing $x_-^{(j)}$ and $x_+^{(k)}$, we denote by $\scrM_{\mathrm{borel}}(x_-^{(j)},x_+^{(k)})$ the moduli space of pairs $(v,u)$ mod common translation by $\R$, excluding as usual the trivial case where both $v$ and $u$ are constant. Provided that the Morse functions and metrics have been chosen generically, these moduli spaces are smooth with
\begin{equation}
\mathrm{dim}\, \scrM_{\mathrm{borel}}(x_-^{(j)},x_+^{(k)}) =
i_M(x_+^{(k)}) - i_M(x_-^{(j)}) + k - j - 1,
\end{equation}
where $i_M$ is the Morse index taken in the fibre. Moreover, there are compactifications ${\bar\scrM}_{\mathrm{borel}}(x_-^{(j)},x_+^{(k)})$ by broken solutions, with the same properties as in ordinary Morse cohomology. We then define $d_{\mathrm{borel}}: C^*_{\mathrm{borel}} \rightarrow C^{*+1}_{\mathrm{borel}}$ as usual by
\begin{equation}
d_{\mathrm{borel}}(x_-^{(j)}) = \sum_k \sum_{x_+^{(k)}}
\#\scrM_{\mathrm{borel}}(x_-^{(j)},x_+^{(k)})\, x_+^{(k)},
\end{equation}
where $\#$ is counting isolated points mod $2$. The cohomology of this complex is isomorphic to ordinary equivariant cohomology (the proof is not difficult, but we will not describe it here).

The construction above leaves a lot of freedom of choice, and that can even be a little confusing. For instance, the structure of equivariant cohomology as a module over $H^*(\RP{\infty};\Z_2) = \Z_2[q]$ is not immediately visible. To improve the situation, note that there is an isometric embedding of $\RP{\infty}$ into itself, namely the infinite shift $\tau(z_0,z_1,z_2,\cdots) = (0,z_0,z_1,\cdots)$. This satisfies
\begin{equation}
\tau^*h = h+1,
\end{equation}
hence preserves the gradient flow of $h$. Lift $\tau$ to an embedding $\sigma$ of $M_{\mathrm{borel}}$ into itself, which is a fibrewise isomorphism (there are two ways of doing that, differing by $\iota$; just pick one). It then makes sense to ask that the family of Morse functions should be invariant under this, in the sense that $f_z = \sigma^*(f_{\tau(z)})$, and the same for the metrics. In particular, all the $f^{(k)}$ can be identified with a fixed Morse function $f = f^{(0)}$ on $M_{\mathrm{borel},0} = M$, so that
\begin{equation} \label{eqn:e2page}
C^*_{\mathrm{borel}} = C^*[[q]] = C^* \otimes_{\Z_2} \Z_2[[q]],
\end{equation}
where $C^*$ is the Morse cochain space of $f$, and $q$ is a formal variable of degree $1$ (the distinction between polynomials $C^*[q]$ and formal power series $C^*[[q]]$, just like that between a direct sum and product in \eqref{eq:borel-complex}, is irrelevant because only finitely powers of $q$ appear in any given degree; we choose the formal power series notation because that's what's appears naturally in more general contexts, compare the discussion in \cite{jones}). The generator $xq^k$ of this corresponds to our previous $x^{(k)}$, namely $x$ considered as a critical point of $f^{(k)}$.

A little reflection shows that this still leaves enough freedom for all the necessary transversality properties to hold (the same issue will appear later in Floer theory, where the argument is entirely parallel, cf. Section \ref{Section:equivariant}). Namely, consider flow lines $v$ of $h$ going from $z^{(0)}$ to $z^{(k)}$ for some $k>0$. The invariance condition only constrains the choice of $f_{v(s)}$ for $s \gg 0$, where $v$ converges to $z^{(k)}$. Hence, it still leaves enough freedom to make the moduli spaces of solutions $u$ lying over such $v$ regular. Invariance then implies that all other spaces are regular as well. In fact, these moduli spaces only depend on $x_{\pm} \in M$ and the difference $k-j$, and we'll therefore write them as $\scrM_{\mathrm{borel}}(x_-,x_+)^{(k-j)}$. The differential is then the $\Z_2[[q]]$-linear map defined by
\begin{equation} \label{eq:first-borel}
d_{\mathrm{borel}}(x_-) = d^{(0)} + d^{(1)}q + d^{(2)}q^2 + \cdots = \sum_i \Big(
\sum_{x_+} \#\scrM_{\mathrm{borel}}(x_-,x_+)^{(i)} \, x_+\Big)q^i.
\end{equation}
It may be helpful to write down the first couple of terms in more concrete terms. The start of the construction is the choice of $f$, and $d^{(0)} = d$ is the standard Morse cohomology differential associated to that function. At the next level, $h$ has two gradient flow lines $v^{(1,\pm)}$ leading from $z^{(0)}$ to $z^{(1)}$. Correspondingly, we should choose two families of Morse functions $f^{(1),+}$ and $f^{(1),-}$ depending on a parameter $s \in \R$, satisfying
\begin{equation}
f^{(1),\pm}_s = f \text{ for $s \ll 0$;} \qquad f^{(1),+}_s = f, \;
f^{(1),-}_s = \iota^*f \text{ for $s \gg 0$.}
\end{equation}
Consider the associated continuation maps. In the $+$ case, the endpoints of the family $f^{(1),+}$ are the same, which means that the continuation map is chain homotopic to the identity. It will be strictly equal to the identity if one chooses $f^{(1),+}$ constant, and we will assume from now on that this is the case. After identifying the Morse complexes of $f$ and $\iota^*f$ in the obvious way, the other continuation map also becomes an endomorphism of $C^*$. This is the Morse-theoretic realization of the action of $\iota$ on cohomology, and we will denote it by $\iota_{\mathrm{morse}}$. One defines the $q^1$ component of \eqref{eq:first-borel} to be\footnote{Even though we work with $\Z_2$-coefficients, signs may occasionally appear in our formulae. These are intended as an aid to the reader's intuition, and may of course be ignored.}
\begin{equation} \label{eq:d-1}
d^{(1)} = \iota_{\mathrm{morse}} - \mathrm{id}.
\end{equation}
Going beyond that, we have two one-parameter families of gradient flow lines of $h$ going from $z^{(0)}$ to $z^{(2)}$. The first family connects the broken flow lines $(v^{(1),-},\tau(v^{(1),+}))$ and $(v^{(1),+},\tau(v^{(1),-}))$. In principle, one should associate to this a one-parameter family of continuation map equations, which then gives rise to a chain homotopy from $\iota_{\mathrm{morse}}$ to itself; however, in analogy with the previous step, a careful choice of the relevant functions will ensure that the resulting moduli space has no isolated points, so that the chain homotopy is zero. The second family connects $(v^{(1),+},\tau(v^{(1),+}))$ and $(v^{(1),-},\tau(v^{(1),-}))$, hence gives rise to a chain homotopy between $\iota_{\mathrm{morse}}^2$ and the identity, which will be the next component $d^{(2)}$ of the equivariant differential. The following term can be thought of as a secondary homotopy between $d^{(2)} \circ \iota_{\mathrm{morse}}$ and $\iota_{\mathrm{morse}} \circ d^{(2)}$, but at even higher orders the picture becomes too complicated to admit an intuitive interpretation.

\subsection{Invariant Morse functions\label{sec:invariant-morse}}
There is an alternative approach leading to a somewhat smaller chain complex than \eqref{eqn:e2page}, which however assumes equivariant transversality. Suppose from now on that our manifold $M$ comes with a $G$-invariant Morse function $f$ and an invariant Riemannian metric. Note that the restriction of $f$ to the fixed point set $M^{\mathrm{inv}} \subset M$ is automatically again Morse. We will further assume that the following conditions are satisfied:
\begin{itemize}
\item
There is a constant $i_{\mathrm{anti}}$, called the normal index of $f$, such that for every $\iota$-invariant critical point, the Morse indices in $M$ and $M^{\mathrm{inv}}$ are related by
\begin{equation} \label{eq:normal-index}
i_M(x) = i_{M^{\mathrm{inv}}}(x) + i_{\mathrm{anti}}.
\end{equation}
\item
The gradient flow of $f$ and that of $f|M^{\mathrm{inv}}$ both satisfy the Morse-Smale condition.
\end{itemize}
The two conditions are logically independent, but still related (since the first one implies that moduli spaces of gradient flow lines in $M$ and $M^{\mathrm{inv}}$ have the same dimension; if the dimension in $M^{\mathrm{inv}}$ was bigger, the Morse-Smale conditions would be contradictory).

\begin{example} \label{th:trough}
For a given $M$ and $\iota$, one can always find a function $f$ which satisfies the condition above. Namely, start with a Morse-Smale function on $M^{\mathrm{inv}}$; extend it to a tubular neighbourhood of $M^{\mathrm{inv}}$ in $M$ by adding a large positive definite quadratic form in normal directions (which means that $i_{\mathrm{anti}} = 0$); use a partition of unity to patch it together with a given invariant Morse function on $M \setminus M^{\mathrm{inv}}$; and then perturb the result slightly if necessary, preserving symmetry. Still, the conditions above are nontrivial, in the sense that they do not hold generically in the space of invariant Morse functions. This will become more important when discussing Floer theory, where topological obstructions to equivariant transversality arise.
\end{example}

Before continuing, we need to set up some notation. Given critical points $x_-$ and $x_+$ of $f$, let $\scrM(x_-,x_+)$ be the standard Morse-theoretic moduli space of unparametrized gradient flow lines, excluding constant ones, and $\bar\scrM(x_-,x_+)$ the compactification by broken flow lines. We also want to consider trajectories whose limits lie in a given $G$-orbit $G x_\pm = \{x_\pm,\iota(x_\pm)\}$. The resulting moduli space, denoted by $\scrM(G x_-,G x_+)$, carries a $G$-action. We write $\scrM(G x_-,G x_+)^{\mathrm{inv}}$ for the fixed part, and $\scrM(G x_-,G x_+)^{\mathrm{non}}$ for the quotient of the free part. The action extends to the compactification $\bar\scrM(G x_-,G x_+)$, and we analogously define spaces $\bar\scrM(G x_-,G x_+)^{\mathrm{inv}}$, $\bar\scrM(G x_-,G x_+)^{\mathrm{non}}$. More concretely, note that if for instance $x_-$ is not $\iota$-invariant, then the $G$-action is free on the whole of $\scrM(G x_-,Gx_+)$, and the quotient can be identified with $\scrM(x_-,x_+) \cup \scrM(x_-,\iota(x_+))$. This also holds for the compactification, and correspondingly in the case where $x_+$ is not $\iota$-invariant. In the remaining situation where both limit points are invariant, so that $\scrM(G x_-,G x_+) = \scrM(x_-,x_+)$, one can identify the fixed point set $\scrM(x_-,x_+)^{\mathrm{inv}}$ with the space of flow lines in $M^{\mathrm{inv}}$, and similarly for the compactification. A less trivial observation is that in this situation, the Morse-Smale conditions and \eqref{eq:normal-index} imply that $\bar\scrM(x_-,x_+)^{\mathrm{inv}}$ is a submanifold of codimension zero in $\bar\scrM(x_-,x_+)$, hence open and closed. Therefore, $\bar\scrM(x_-,x_+)^{\mathrm{non}}$ is a compact manifold with corners\footnote{\label{foot}It is a well-known problem that such compactifications do not have a canonical smooth structure in transverse direction to the boundary strata, which means that they are not quite ``manifold with corners'' in the usual $\smooth$ sense. This means that some care needs to be taken when using standard differential topology constructions.}, being a quotient of a free $G$-action on the compact manifold with corners $\bar\scrM(x_-,x_+) \setminus \bar\scrM(x_-,x_+)^{\mathrm{inv}}$.

Let $C^*_{\mathrm{inv}}$ be the graded vector space over $\Z_2$ freely generated by invariant critical points, with the grading $\mathrm{deg}(x) = i_{M^{\mathrm{inv}}}(x) = i_M(x) - i_{\mathrm{anti}}$. Similarly let $C^*_{\mathrm{non}}$ be the space generated by pairs $G x = \{x,\iota(x)\}$ of non-invariant critical points. One then has natural maps
\begin{equation}
\begin{aligned}
 & d_{\mathrm{inv}}: C^*_{\mathrm{inv}} \longrightarrow C_{\mathrm{inv}}^{*+1}, \\
 & d_{\mathrm{non}}: C^*_{\mathrm{non}} \longrightarrow C_{\mathrm{non}}^{*+1}, \\
 & D_1: C^*_{\mathrm{non}} \longrightarrow C^{*-i_{\mathrm{anti}}+1}_{\mathrm{inv}}, \\
 & D_2: C^{*-i_{\mathrm{anti}}}_{\mathrm{inv}} \longrightarrow C^{*+1}_{\mathrm{non}},
\end{aligned}
\end{equation}
defined as follows. $d_{\mathrm{inv}}$ is the Morse cohomology differential for $f|M^{\mathrm{inv}}$, hence counts the number of isolated $\iota$-invariant connecting orbits. Since we use $\Z_2$-coefficients, and the non-invariant connecting orbits between invariant critical points come in pairs, we may just as well count all critical orbits, meaning that
\begin{equation}
 d_{\mathrm{inv}}(x_-) = \sum_{x_+} \# \scrM(x_-,x_+)\, x_+
\end{equation}
where the sum is over all invariant $x_+$. Next, connecting trajectories between non-invariant critical points come in pairs exchanged by $\iota$. The map $d_{\mathrm{non}}$ counts each pair once (formally, one can see this as the Morse differential for the induced function on the quotient $(M \setminus M^{\mathrm{inv}})/G$):
\begin{equation}
 d_{\mathrm{non}}(G x_-) = \sum_{G x_+} \#\scrM(G x_-,G
 x_+)^{\mathrm{non}}\, G x_+.
\end{equation}
Similarly $D_2$ counts connecting trajectories from invariant to non-invariant critical points, and vice versa for $D_1$, again using the spaces $\scrM(G x_-,G x_+)^{\mathrm{non}}$.

To introduce yet another piece of data, we pick one out of each pair of non-invariant critical points, and denote the resulting set of preferred points by $P$. One can then count only those isolated connecting trajectories which go from a preferred point to a non-preferred one, which yields another map
\begin{equation}
 U: C^*_{\mathrm{non}} \longrightarrow C^{*+1}_{\mathrm{non}}.
\end{equation}
The Morse complex for $f$, in the ordinary sense, can now be written as $C^* = C^*_{\mathrm{non}} \oplus C^*_{\mathrm{non}} \oplus C^{*-i_{\mathrm{anti}}}_{\mathrm{inv}}$, with the differential
\begin{equation} \label{eq:d-morse}
d = \begin{pmatrix}
 d_{\mathrm{non}} & 0 & 0 \\
 U & d_{\mathrm{non}} & D_2 \\
 D_1 & 0 & d_{\mathrm{inv}}
\end{pmatrix}.
\end{equation}
Here, the first copy of $C^*_{\mathrm{non}}$ is identified with the subspace of $C^*$ generated by the critical points $x \in P$, whereas the second copy is generated by sums $x + \iota(x)$; this explains the asymmetry in \eqref{eq:d-morse}. The relation $d^2 = 0$ yields various equations between its components; some which will be relevant later on are
\begin{equation} \label{eq:slurp}
\begin{aligned}
 d_{\mathrm{non}} D_2 + D_2 d_{\mathrm{inv}} & = 0, \\
 d_{\mathrm{non}} U + U d_{\mathrm{non}} &  = D_2D_1.
\end{aligned}
\end{equation}

We can now introduce the equivariant Morse complex. Algebraically, this is defined by starting with the canonical induced involution on $C^*$, namely
\begin{equation} \label{eq:iota-morse}
\iota_{\mathrm{morse}}(a,b,c) = (a,a+b,c),
\end{equation}
and the standard free resolution of the trivial $\Z_2[G]$-module, which is
\begin{equation} \label{eq:abstract-borel}
 P^* = \{ \cdots \longrightarrow \Z_2[G] \xrightarrow{\mathrm{id}+\iota}
 \Z_2[G] \xrightarrow{\mathrm{id}-\iota} \Z_2[G]
 \longrightarrow 0\},
\end{equation}
and forming $\mathrm{Hom}_{\Z_2[G]}(P^*,C^*)$. Concretely, the outcome is $C^*_{\mathrm{borel}} = C^*_{\mathrm{non}}[[q]] \oplus C^*_{\mathrm{non}}[[q]] \oplus C^{*-i_{\mathrm{anti}}}_{\mathrm{inv}}[[q]]$ with
\begin{equation} \label{eq:borel}
d_{\mathrm{borel}} = d + (\iota_{\mathrm{morse}}-\mathrm{id}) q =
\begin{pmatrix} d_{\mathrm{non}} & 0 & 0 \\
U+q & d_{\mathrm{non}} & D_2 \\
D_1 & 0 & d_{\mathrm{inv}}
\end{pmatrix}.
\end{equation}
Here $d_{\mathrm{non}}$, $d_{\mathrm{inv}}$, $D_1$, $D_2$, $U$ have been extended to $\Z_2[[q]]$-linear maps in the obvious way. To make the connection with the preceding, more general, construction, observe that in the present context, one can extend $f$ in a locally constant way to a family of Morse functions on the fibres of $M_{\mathrm{borel}}$. With these choices, our previous definition of $\iota_{\mathrm{morse}}$ agrees with \eqref{eq:iota-morse}, and the higher order terms all vanish, which means that \eqref{eq:first-borel} indeed reduces to \eqref{eq:borel}.

\begin{lemma} \label{th:acyclic}
The elements $(a,0,0)$, $a \in C^*_{\mathrm{non}}[[q]]$, together with their $d_{\mathrm{borel}}$-images, span an acyclic subcomplex $Z^* \subset C^*_{\mathrm{borel}}$. Moreover, the inclusion $j: C^*_{\mathrm{non}} \oplus C^{*-i_{\mathrm{anti}}}_{\mathrm{inv}}[[q]] \rightarrow C^*_{\mathrm{borel}}$, $(b,c) \mapsto (0,b q^0,c)$ descends to an isomorphism of graded vector spaces
\begin{equation} \label{eq:reduce}
C^*_{\mathrm{non}} \oplus C^{*-i_{\mathrm{anti}}}_{\mathrm{inv}}[[q]] \longrightarrow
C^*_{\mathrm{borel}}/Z^*.
\end{equation}
\end{lemma}

\begin{proof} Since $U$ counts gradient flow lines, it increases the value of our Morse function, hence is nilpotent (of course, one can see this also just by looking at the grading). As a consequence, the map
\begin{equation}
U+q: C^*_{\mathrm{non}}[[q]] \longrightarrow C^*_{\mathrm{non}}[[q]]/C^*_{\mathrm{non}}q^0
\end{equation}
is an isomorphism. This implies that the restriction of $d_{\mathrm{borel}}$ to the subspace $(a,0,0)$ is injective, hence that $Z^*$ is acyclic, and also the bijectivity of \eqref{eq:reduce}.
\end{proof}

We write $C^*_{\mathrm{equiv}} = C^*_{\mathrm{non}} \oplus C^{*-i_{\mathrm{anti}}}_{\mathrm{inv}}[[q]]$ and equip it with the differential and $\Z_2[[q]]$-module structure induced from \eqref{eq:reduce}. The degree $1$ endomorphism which generates the module structure has quite a simple form, namely
\begin{equation}
Q_{\mathrm{equiv}} = \begin{pmatrix} U & 0 \\ D_1 & q \end{pmatrix}.
\end{equation}
To write down the differential, we need some more notation. In general, given any free graded $\Z_2[[q]]$-module $A^*[[q]]$, one has a map $|_{q=0}: A^*[[q]] \rightarrow A^* \subset A^*[[q]]$ which forgets all terms of positive $q$-order, and an endomorphism $\partial_q: A^*[[q]] \rightarrow A^{*-1}[[q]]$ which kills the constant term and divides each higher order term by $q$. These are related by $q\partial_q = \mathrm{id} - |_{q=0}$. By $(\mathrm{id} + U \partial_q)^{-1}$ we denote the geometric series $\mathrm{id} + U \partial_q + U^2 \partial_q^2 + \cdots$, which terminates because $U$ is nilpotent. In this terminology, the outcome of a straightforward computation is that
\begin{equation}
d_{\mathrm{equiv}} = \begin{pmatrix} d_{\mathrm{non}} & |_{q=0} (\mathrm{id} + U \partial_q)^{-1} D_2 \\ 0 & d_{\mathrm{inv}} + D_1 (\mathrm{id} + U \partial_q)^{-1} D_2 \partial_q
\end{pmatrix},
\end{equation}
where $D_1,D_2,U$ have been extended to $\Z_2[[q]]$-linear homomorphisms between $C^*_{\mathrm{inv}}[[q]]$, $C^*_{\mathrm{non}}[[q]]$ (this kind of extension will be made without explicit mention from now on).

\subsection{The localization map\label{sec:localisation}}
Each $\bar\scrM = \bar\scrM(G x_-,G x_+)^{\mathrm{non}}$ is a free $G$-quotient, hence carries an obvious double cover. Denote by $\xi(G x_-, G x_+) = \xi \rightarrow \bar\scrM$ the associated real line bundle. Following a well-known strategy, we will extract additional information from these moduli spaces by intersecting the zero-sets of sufficiently many sections of our line bundles. For technical reasons, we will use sections which are continuous overall, and smooth on each stratum (the interior as well as boundary strata). Obviously, one needs to check that the standard differential-topology arguments go through in this context. Namely, suppose that we are given (a finite or infinite number of) sections $\sigma^{(1)},\sigma^{(2)},\dots$ of this type, satisfying the following transversality condition for all $k$:
\begin{equation} \label{eq:s-cut-down}
\begin{minipage}{30em} \em
For each stratum $\scrS \subset \bar\scrM$, the restriction of $\sigma^{(k)}$ to $\scrS \cap (\sigma^{(1)})^{-1}(0) \cap \cdots \cap (\sigma^{(k-1)})^{-1}(0)$ is transverse to the zero-section.
\end{minipage}
\end{equation}
(this makes sense when seen as inductive in $k$, since the condition for $k-1$ will ensure that the intersection in \eqref{eq:s-cut-down} is smooth). We then write $\bar\scrM^{(k)}$ for the intersection of the zero-sets of all the $\sigma^{(i)}$, $1 \leq i \leq k$, and $\scrM^{(k)}$ for the intersection of this with the interior $\scrM$. Consider first the case when $k$ is the dimension of the space $\scrM$ itself. Since the boundary strata $\scrS$ are manifolds of smaller dimension, the transversality assumption ensures that $\bar\scrM^{(k)} \cap \scrS = \emptyset$, which means that $\scrM^{(k)} = \bar\scrM^{(k)}$ is a finite subset of the interior $\scrM$. We will then extract algebraic information as usual, by counting the number of these points modulo $2$ in that subset. Next, assume that $\scrM$ is of dimension $k+1$, and define
\begin{equation} \label{eq:fake-boundary}
\partial\bar\scrM^{(k)} = \bar\scrM^{(k)} \cap \partial\bar\scrM = \bar\scrM^{(k)} \setminus \scrM^{(k)}.
\end{equation}
In spite of the notation, it is by no means clear that this is a boundary in the standard manifold sense. However, \eqref{eq:s-cut-down} at least ensures that \eqref{eq:fake-boundary} consists of finitely many points, which all lie inside the top (codimension 1) boundary strata. Consider the local
picture near such a point, using a local chart $\R^{\geq 0} \times \R^k \rightarrow \bar\scrM$ provided by a suitable gluing map, and a local trivialization of $\xi$ over it. In this chart, the $\sigma^{(i)}$, $1 \leq i \leq k$, together give a map
\begin{equation}
s = (\sigma^{(1)},\dots,\sigma^{(k)}): \R^{\geq 0} \times \R^k \longrightarrow \R^k, \quad s(0) = 0.
\end{equation}
By construction, $s$ is continuous, and smooth on both strata $\{0\} \times \R^k$ and $\R^{>0} \times \R^k$, with regular zero-sets in each. In particular, the restriction of $s$ to any sufficiently small sphere $\{0\} \times S^{k-1}_\epsilon$ is a map $S^{k-1}_\epsilon \rightarrow \R^k \setminus \{0\}$ of degree $\pm 1$. If we then consider the hemisphere $S^k_\epsilon \cap (\R^{\geq 0} \times \R^k)$, the restriction of $s$ to that hemisphere must still have an odd number of zeros, counted algebraically (this only uses continuity of $s$ up to the boundary). By choosing $\epsilon$ generically, one can ensure that $s^{-1}(0)$ intersects the hemisphere transversally, hence the actual number of zeros is also odd. Now assume that we remove from $\bar\scrM^{(k)}$ its intersection with such hemispheres around each point of $\partial\bar\scrM^{(k)}$. The outcome is a compact one-manifold with boundary, which by the local considerations above has an odd number of boundary points for each point of $\partial\bar\scrM^{(k)}$, proving that that space itself consists of an even number of points. In this slightly roundabout way, one arrives at the same result as in the more familiar case of sections extending smoothly to the compactification.

The moduli spaces $\bar\scrM$ and their line bundles have certain inductive relations to each other, and the sections must be chosen accordingly. To see what this means, take a point of $\partial\bar\scrM$. This has the form $G\bar{u} = \{\bar{u},\iota(\bar{u})\}$, where $\bar{u} = (u_1,\dots,u_k)$ is a broken trajectory with $k \geq 2$ pieces (the pieces are $u_j \in \scrM(x_{j-1},x_j)$ for some critical points $x_j$ such that $x_0 \in G x_-$, $x_k \in G x_+$). From our previous discussion of the compactification, we know that at least one of the pieces $u_j$ is not $\iota$-invariant. Suppose that $j \in \{1,\dots,k\}$ is the largest number with that property. Clearly, choosing one of the two elements in $G \bar{u}$ is equivalent to choosing one of the two $\{u_j,\iota(u_j)\}$. Hence, there is a canonical isomorphism
\begin{equation}
\label{eq:boundary-iso}
\xi_{G\bar{u}} \iso \xi(G x_{j-1},G x_j)_{G u_j}.
\end{equation}
Viewed more globally, the boundary stratum $\scrS \subset \partial \bar\scrM$ to which $G \bar{u}$ belongs has a natural projection map to $\scrM(G x_{j-1},G x_j)^{\mathrm{non}}$, which is such that $\xi|\scrS$ is canonically isomorphic to the pullback of $\xi(G x_{j-1},G x_j)$. In this form, the statement extends to the closure $\bar{\scrS}$, which projects to $\bar\scrM(G x_{j-1},G x_j)^{\mathrm{non}}$.

\begin{definition}
Let $\sigma = \{\sigma(G x_-,G x_+): \bar\scrM(G x_-,G x_+) \rightarrow \xi(G x_-,G x_+)\}$ be a family of sections, one for each pair $(G x_-,G x_+)$. We call $\sigma$ consistent if it is compatible with \eqref{eq:boundary-iso}, in the sense that $\sigma(G x_-,G x_+)_{G\bar{u}}$ equals $\sigma(G x_{j-1},G x_j)_{G u_j}$ under that isomorphism.
\end{definition}

The consistency condition prescribes the behaviour of $\sigma(G x_-,G x_+)$ on each stratum of the boundary, in terms of the sections over smaller-dimensional moduli spaces. One can easily check that the consistency conditions for various strata do not contradict each other, so that consistent families of sections can be constructed inductively. In fact, the freedom in the choice of such a family is quite large, so standard transversality theory applies. Specifically, let $\sigma^{(1)},\sigma^{(2)},\dots$ be an infinite sequence of consistent families of sections. We say that this sequence is regular if its restriction to any fixed moduli space satisfies \eqref{eq:s-cut-down} for all $k$. An easy transversality argument shows that regularity is a generic
property.

Suppose that we have chosen a regular consistent sequence. As before, we can then consider the cut-down moduli spaces $\scrM^{(k)} = \scrM(G x_-, G x_+)^{\mathrm{non},(k)}$ and their compactifications. Define
\begin{equation}
\begin{aligned}
& D_1^{(k)}: C^*_{\mathrm{non}} \longrightarrow C^{*-i_{\mathrm{anti}}+k+1}_{\mathrm{inv}}, \\
& D_1^{(k)}(G x_-) = \sum_{x_+} \# \scrM(Gx_-,x_+)^{\mathrm{non},(k)}\, x_+.
\end{aligned}
\end{equation}
Of course, only finitely many of these are nonzero. To simplify the notation, we arrange them into a power series whose ``variable'' is the operation $\partial_q$:
\begin{equation}
\begin{aligned}
& \DD_1: C^*_{\mathrm{non}}[[q]] \longrightarrow C^{*-i_{\mathrm{anti}}+1}_{\mathrm{inv}}[[q]], \\
& \DD_1 = D_1^{(0)} + D_1^{(1)} \partial_q + D_1^{(2)} \partial_q^2 + \cdots.
\end{aligned}
\end{equation}
In the same way, but this time using the case where both $x_-$ and $x_+$ are invariant, one defines operations $X^{(k)}: C^*_{\mathrm{inv}} \rightarrow C^{*+k+1}_{\mathrm{inv}}$ and packages them into a series $\XX: C^*_{\mathrm{inv}}[[q]] \rightarrow C^{*+1}_{\mathrm{inv}}[[q]]$. Note that while $D_1^{(0)} = D_1$ as defined previously, $X^{(0)}$ is new: it counts the pairs of non-invariant flow lines connecting invariant critical points, which do not contribute to the ordinary differential $d$.

\begin{remark}
One can legitimately wonder why there are no higher order maps $D_2^{(k)}$ associated to moduli spaces which go from invariant $x_-$ to non-invariant $x_+$. The reason lies in the notion of consistency adopted here: any broken flow line in $\bar\scrM(G x_-,G x_+)^{\mathrm{non}}$ has the property that its last piece is non-invariant. Hence, choosing a representative $x_+$ trivializes the line bundle $\xi(G x_-, G x_+)$ in a way which is compatible with \eqref{eq:boundary-iso}, which means that no nontrivial information can be obtained in this way.

Of course, one could reverse the conventions by choosing the smallest possible $j$ in \eqref{eq:boundary-iso}. That would lead to a different algebraic formalism, with trivial maps $D_1^{(k)}$ and nontrivial maps $D_2^{(k)}$. Ultimately, one expects the resulting localization map to be the same up to homotopy, but we have not checked that this is the case.
\end{remark}

\begin{lemma} \label{th:antislurp}
$d_{\mathrm{inv}} \XX + \XX d_{\mathrm{inv}} = \DD_1 D_2$.
\end{lemma}

In components, this means that $d_{\mathrm{inv}} X^{(k)} + X^{(k)} d_{\mathrm{inv}} = D_1^{(k)} D_2$ for each $k$. The proof is standard: one looks at those $x_\pm$ such that $i_M(x_+) - i_M(x_-) = k+2$, which is where the space $\scrM(x_-,x_+)^{\mathrm{non},(k)}$ is one-dimensional, and counts its boundary points in the sense of \eqref{eq:fake-boundary}. The only case worth mentioning is that of broken trajectories of the form $(u_1,u_2) \in \bar\scrM(x_-,x_1) \times \bar\scrM(x_1,x_+)$, where $x_1$ is an invariant critical point with $i_M(x_1) = i_M(x_-) + 1$, and both $u_1,u_2$ are non-invariant trajectories, the $G$-orbit of the latter lying in the subspace $\{\sigma^{(1)} = \cdots = \sigma^{(k)} = 0\}$. Such a broken trajectory gives rise to a pair of boundary points in our moduli space, namely $G (u_1,u_2)$ and $G (\iota(u_1),u_2)$. The contributions of these two points cancel, and therefore our formula does not contain a term corresponding to this boundary stratum.

The situation where $x_+$ is $\iota$-invariant, but $x_-$ is not, can be exploited a little further. Suppose that $x_- \in P$. Then, as mentioned before, $\bar\scrM(x_-,x_+)$ projects isomorphically to the quotient $\bar\scrM = \bar\scrM(G x_-,x_+)^{\mathrm{non}}$. This provides a canonical nowhere zero section of the associated line bundle, which we denote by $\sigma^{\mathrm{pref}} = \sigma(G x_-,x_+)^{\mathrm{pref}}$ (these sections do not normally satisfy the consistency condition above). Define a subset of $\scrM = \scrM(G x_-,x_+)^{\mathrm{non}}$ by
\begin{equation} \label{eq:geq-space}
\scrM^{(k,\geq)} = \scrM(G x_-,x_+)^{\mathrm{non},(k,\geq)} = \{\sigma^{(1)} = \cdots = \sigma^{(k)} = 0, \; \sigma^{(k+1)}/\sigma^{\mathrm{pref}} \geq 0 \},
\end{equation}
and similarly for the compactification, $\bar\scrM^{(k,\geq)} \subset \bar\scrM$. In the case where $\dim\,\scrM = k$, there are no points in \eqref{eq:geq-space} where $\sigma^{(k+1)} = 0$, and also no points at infinity in the compactification. By counting points in these spaces, one gets new operations
\begin{equation}
S_1^{(k)} : C^*_{\mathrm{non}} \longrightarrow C^{*-i_{\mathrm{anti}}+k+1}_{\mathrm{inv}},
\end{equation}
which we again write as a single power series $\SS_1$. As an example, suppose that $\scrM(G x_-,x_+)^{\mathrm{non}}$ is zero-dimensional, so that $k = 0$. Then, since $\sigma^{(1)}$ is nowhere zero, it picks out a representative for each orbit $G u$ in the moduli space. $S_1^{(0)}$ counts only those $G u$ for which $\sigma^{(1)}/\sigma^{\mathrm{pref}} > 0$, which more concretely means that the representative selected by $\sigma^{(1)}$ starts at the preferred point $x_-$. The other relevant situation is when $\dim\,\scrM = k+1$. In that case, \eqref{eq:geq-space} is a smooth one-manifold whose boundary is $\scrM^{(k+1)}$, and compactifying adds finitely many points at infinity, which can be treated in the same way as \eqref{eq:fake-boundary}. In analogy with the previous Lemma, this leads to the following:

\begin{lemma} \label{th:slaver}
$d_{\mathrm{inv}}\SS_1 + \SS_1 d_{\mathrm{non}} = \XX D_1 + \DD_1 U + (\DD_1-D_1)/\partial_q$, where the formal quotient in the last term stands for $(\DD_1-D_1)/\partial_q = D_1^{(1)} + D_1^{(2)}\partial_q + D_1^{(3)} \partial_q^2 + \cdots$.
\end{lemma}

Again, this is best understood by writing out the equation in components,
\begin{equation}
d_{\mathrm{inv}} S_1^{(k)} + S_1^{(k)} d_{\mathrm{non}} = D_1^{(k)} U + X^{(k)} D_1 +
D_1^{(k+1)}.
\end{equation}
The most interesting terms are $D_1^{(k)} U$ and $S_1^{(k)} d_{\mathrm{non}}$, which appear together in the following way. Take $x_-,x_+$ such that $i_M(x_+) - i_M(x_-) = k+2$; we assume as before that the representative $x_-$ is chosen to lie in $P$. Take a pair $\bar{u} = (u_1,u_2) \in \scrM(x_-,x_1) \times \scrM(x_1,x_+)$, where $x_1$ is a non-invariant critical point, and $G u_2 \in \scrM(x_1,x_+)^{\mathrm{non},(k)}$. By the transversality assumptions, $\sigma^{(k+1)}(G u_2) \neq 0$. If $x_1 \in P$, then $G\bar{u}$ will lie on the boundary of $\bar\scrM(x_-,x_+)^{\mathrm{non},(k,\geq)}$ iff $\sigma^{(k+1)}(G u_2) > 0$. The total contribution of such points to our formula is $S_1^{(k)} (d_{\mathrm{non}} - U)$. In the other case $x_1 \notin P$, $G\bar{u}$ lies on the boundary iff $\sigma^{(k+1)}(G u_2) < 0$, and this contributes $(D_1^{(k)} - S_1^{(k)}) U$. Adding up the two terms yields the desired $D_1^{(k)} U + S_1^{(k)} d_{\mathrm{non}}$.

Consider the map
\begin{equation} \label{eq:lambda-map}
\begin{aligned}
& \lambda: (C^*_{\mathrm{equiv}},d_{\mathrm{equiv}}) \longrightarrow (C^{*-i_{\mathrm{anti}}}_{\mathrm{inv}}[[q]],d_{\mathrm{inv}}), \\
& \lambda(b,c) = c + \XX \partial_q(c) + \SS_1 (\mathrm{id} + U \partial_q)^{-1} D_2 \partial_q^2(c).
\end{aligned}
\end{equation}
This makes sense since $U$ is nilpotent, which together with our previous observations implies that the formula contains only finitely many powers of $\partial_q$. Direct computation using \eqref{eq:slurp} and Lemmas \ref{th:antislurp}, \ref{th:slaver} shows that $\lambda$ is a chain homomorphism. Combined with the canonical projection from $C_{\mathrm{borel}}^*$ to $C_{\mathrm{equiv}}^*$, it induces a homomorphism
\begin{equation} \label{eqn:localisationmap}
 \Lambda: H^*(C_{\mathrm{borel}},d_{\mathrm{borel}}) \longrightarrow
 H^{*-i_{\mathrm{anti}}}(C_{\mathrm{inv}},d_{\mathrm{inv}})[[q]].
\end{equation}
We call $\Lambda$ the bare localization map.

\begin{lemma} \label{th:rescale-map}
For sufficiently large $m$, the composition $\lambda
Q_{\mathrm{equiv}}^m$ is a map of $\Z_2[[q]]$-modules. This means that
\begin{equation}
q \lambda Q_{\mathrm{equiv}}^m = \lambda Q_{\mathrm{equiv}}^{m+1},
\end{equation}
and similarly for all formal power series in $q$.
\end{lemma}

\begin{proof}
Since $\lambda$ only contains finitely many powers of $\partial_q$, there is an $l \gg 0$ such that $q \lambda(0,c) - \lambda (0,qc)$ vanishes for $c \in q^l C^*_{\mathrm{inv}}[[q]]$. Using again the nilpotency of $U$, one can choose $m$ in such a way that $Q_{\mathrm{equiv}}^m(C^*_{\mathrm{equiv}}) \subset \{0\} \oplus q^l C^*_{\mathrm{inv}}[[q]]$, and then obviously $q \lambda Q_{\mathrm{equiv}}^m = \lambda\,\diag(0,q)\,Q_{\mathrm{equiv}}^m = \lambda Q_{\mathrm{equiv}}^{m+1}$.
\end{proof}

For $m$ as above, define $\Lambda^{(m)}: H^*(C_{\mathrm{borel}},d_{\mathrm{borel}}) \rightarrow H^{*-i_{\mathrm{anti}}+m}(C_{\mathrm{inv}},d_{\mathrm{inv}})[[q]]$ to be the map induced by $\lambda Q_{\mathrm{equiv}}^m$. We call this the $m$-normalized localization map.

\begin{lemma} \label{th:localization}
After tensoring over $\Z_2[[q]]$ with $\Z_2((q)) = \Z_2[[q]][q^{-1}]$, $\Lambda^{(m)}$ becomes an isomorphism. Equivalently, the kernel of the original $\Lambda^{(m)}$ is the torsion part of $H^*(C_{\mathrm{borel}},d_{\mathrm{borel}})$ as a $\Z_2[[q]]$-module, and its cokernel is finite-dimensional.
\end{lemma}

\begin{proof}
All nontrivial terms in our definitions involve counting gradient flow lines, hence are strictly increasing with respect to the filtration by values of our Morse function. In particular
\begin{equation}
\lambda Q_{\mathrm{equiv}}^m(b,c) = (0,q^m c) + \text{terms which increase the filtration}.
\end{equation}
Hence, this chain map has finite-dimensional kernel and cokernel, and then the same applies to the induced map on cohomology.
\end{proof}

\begin{remark} \label{rem:cocycle}
We want to outline a possible technical alternative. This avoids the choice of sections and the transversality issues involved in cutting down moduli spaces, but it assumes the existence of fundamental chains $[\bar\scrM]$ (in the sense of cubical singular homology), which moreover should be compatible with the product structure of the boundary strata. The construction of such chains in Floer theory is part of the formalism of \cite{FO3}; the Morse case is similar but simpler.

Instead of the line bundle $\xi \rightarrow \bar\scrM = \bar\scrM(G x_-,G x_+)^{\mathrm{non}}$, consider the underlying double cover $S(\xi) = \bar\scrM(G x_-, G x_+) \setminus \bar\scrM(G x_-, G x_+)^{\mathrm{inv}}$. One can define cocycle representatives for the first Stiefel-Whitney class $w_1(\xi)$ as follows. Pick a singular zero-cocycle $v$ on $S(\xi)$, which is simply a function $S(\xi) \rightarrow \Z_2$, such that the fibrewise involution takes $v$ to $1-v$. Its coboundary $dv$ is necessarily invariant under the involution, which means that it is the pullback of a one-cocycle $w$ on $\bar\scrM$. Explicitly, to determine the value of $w$ on a chain $[0,1] \rightarrow \bar\scrM$, one lifts that chain to $S(\xi)$ arbitrarily, and then takes the difference of the values of $v$ at the endpoints, which is independent of the choice of lift. In particular, by looking at closed loops, one sees that $w$ measures their liftability, hence represents $w_1(\xi)$. We assume that cocycles $v$ have been chosen on all of our moduli spaces, compatibly with \eqref{eq:boundary-iso}, and take the resulting cocycles $w$. Then, $D_1^{(k)}$ and $X^{(k)}$ can be defined by evaluating the $k$-th power $w^k$ on the fundamental chain $[\bar\scrM]$.

Consider now the case where $x_-$ is $\iota$-invariant, but $x_+$ is not. In that case, the subspace $\bar\scrM(x_-,x_+) \subset S(\xi)$ provides a canonical section $s$ of $S(\xi)$. To define $S^{(k)}$, we evaluate $s^*(v)w^k$ on the fundamental chain. By construction $d(s^*(v)w^k) = w^{k+1}$, which is the mechanism leading to Lemma \ref{th:slaver} in this framework.
\end{remark}

\subsection{The Smith inequality\label{sec:smith-inequality}}
The decreasing filtration of $C^*_{\mathrm{borel}}$ by powers of $q$ yields a spectral sequence, whose starting page consists of infinitely many copies of ordinary Morse cohomology:
\begin{equation} \label{eq:borel-ss}
 E_1^{pq} = \begin{cases}
 H^q(C,d) & p \geq 0, \\
 0 & p < 0.
 \end{cases}
\end{equation}
As we have seen above, the first differential, in the nontrivial case $p \geq 0$, is $\mathrm{id} - \iota^*: E_1^{pq} \rightarrow E_1^{p+1,q}$. In a similar but technically simpler vein, multiplication by $q$ defines a universal coefficient sequence
\begin{equation} \label{eq:uct}
\cdots H^{*-1}(C_{\mathrm{borel}},d_{\mathrm{borel}}) \xrightarrow{q} H^*(C_{\mathrm{borel}},d_{\mathrm{borel}}) \rightarrow H^*(C,d) \rightarrow \cdots
\end{equation}
Using either approach, one sees that the rank of $H^*(C_{\mathrm{borel}},d_{\mathrm{borel}})$ as a $\Z_2[[q]]$-module is bounded above by the dimension of $H^*(C,d)$. More precisely, from \eqref{eq:uct} one gets the following statement:

\begin{lemma}
Let $r_{\mathrm{free}}$ be the number of generators of the torsion-free part of the $\Z_2[[q]]$-module $H^*(C_{\mathrm{borel}},d_{\mathrm{borel}})$, and $r_{\mathrm{tor}}$ the same for the torsion part.
Then
\begin{equation}
r_{\mathrm{free}} + 2 r_{\mathrm{tor}} = \mathrm{dim}\, H^*(C,d).
\end{equation}
\end{lemma}

On the other hand, by Lemma \ref{th:localization}, the rank of the torsion-free part is isomorphic to that of $H^*(C_{\mathrm{inv}},d_{\mathrm{inv}})[[q]]$. This implies the Smith inequality in the following form:

\begin{lemma} \label{lemma:minismith}
The total dimension of $H^*(C_{\mathrm{inv}},d_{\mathrm{inv}})$ is less or equal than
that of $H^*(C,d)$. Equality holds iff $H^*(C_{\mathrm{borel}},d_{\mathrm{borel}})$ is torsion-free as a $\Z_2[[q]]$-module, or equivalently if the spectral sequence \eqref{eq:borel-ss} degenerates at the $E_1$ term.
\end{lemma}

%\begin{remark}
%Each column $E_k^{p*}$ of the spectral sequence is in fact a
%$\Z_2[q]$-module, and the differentials respect that structure. One
%could recast the argument above in terms of ranks of
%$\Z_2[q]$-modules, and thereby avoid using the grading of homology
%groups (compare Remark \ref{th:alternative}).
%\end{remark}
%

It is possible to extract a little additional information, going beyond the Smith inequality, from the localisation map. We know that the cokernel of $\Lambda^{(m)}$ is a finitely generated torsion $\Z_2[[q]]$-module, hence can be written as
\begin{equation}
\mathrm{coker}(\Lambda^{(m)}) = \bigoplus_k (\Z_2[[q]]/q^{d_k+1}) a_k
\end{equation}
where $a_k$ are homogeneous generators, and $d_k \geq 0$. We define the cokernel polynomial to be
\begin{equation} \label{equation:cokernelpoly}
P_{\mathrm{coker}}(t) = t^{-m} \sum_k (-1)^{\mathrm{deg}(a_k)} t^{d_k} \in \Z[t,t^{-1}].
\end{equation}
This is independent of $m$ since, by construction, $\Lambda^{(m+1)} = q\Lambda^{(m)}$.

\begin{example}
Take $M = \R^n$ (this is noncompact, but that will be irrelevant since the gradient flow is still well-behaved), with $\iota(x) = -x$, and a function $f$ that satisfies $f(x) = -\|x\|^2$ near the origin and $f(x) = \|x\|^2$ outside a compact subset. In this case $i_{\mathrm{anti}} = n$, $H^*(C_{\mathrm{borel}},d_{\mathrm{borel}}) = H^*(C,d)[[q]] = \Z_2[[q]]$, and
\begin{equation}
\Lambda(q^k) = \begin{cases} 0 & k < n, \\ q^{k-n} & k \geq n.
\end{cases}
\end{equation}
This shows clearly why only the rescaled maps $\Lambda^{(m)}$ are well-behaved with respect to the $\Z_2[[q]]$-module structure.
\end{example}

\section{Floer cohomology}

\subsection{Basic setup\label{subsec:borel-floer}}
We will work in a limited framework which remains close to Morse theory, hence makes it unproblematic to carry over our previous constructions. Namely, let $(M,\omega,I)$ be an exact symplectic manifold with a compatible almost complex structure, which is convex at infinity. Exactness means that the symplectic form is the exterior derivative of a distinguished one-form $\theta$. Convexity at infinity says that there is an exhaustion (a sequence of relatively compact open subsets $U_1 \subset U_2 \subset \cdots$ whose union is $M$), such that if $S$ is a connected compact Riemann surface with nonempty boundary, and $u: S \rightarrow M$ an $I$-holomorphic map with $u(\partial S) \subset U_k$, then necessarily $u(S) \subset U_{k+1}$ (symplectic structures obtained from plurisubharmonic exhaustive functions on Stein manifolds satisfy these properties, and so do more general Liouville manifolds). We will consider pairs of Lagrangian submanifolds $(L_0,L_1)$ in $M$ satisfying a relative version of the same conditions. First of all, both $\theta|L_0$ and $\theta|L_1$ are exact one-forms. Secondly, $L_0 \cap L_1$ is compact. Finally, there is an exhaustion of $M$ such that any $I$-holomorphic map $u: S \rightarrow M$, for $S$ as before and with $u(\partial S) \subset L_0 \cup L_1 \cup U_k$, satisfies $u(S) \subset U_{k+1}$. Note that the last two conditions are automatically satisfied if $L_0$, $L_1$ are themselves compact, which is the case in most applications. We allow noncompact Lagrangian submanifolds since they will appear later as a technical tool, as well as in some examples.

The Floer cohomology $HF(L_0,L_1)$ is a vector space over $\Z_2$ (under suitable assumptions, this can be $\Z/m$ graded for some $m$, or even $\Z$ graded, but in this section we'll stick to the general ungraded context). To define it, choose a compactly supported time-dependent Hamiltonian $H \in \smooth_c([0,1] \times M,\R)$, whose associated Hamiltonian isotopy $(\phi^t)$ has the property that $\phi^1(L_0)$ intersects $L_1$ transversally. The (necessarily finitely many) intersection points provide generators for the Floer cochain complex $CF(L_0,L_1)$. For the differential, we choose a generic family $J$ of almost complex structures, which is a time-dependent and compactly supported perturbation of $I$. Given two intersection points $x_\pm$, we denote by $\scrM(x_-,x_+)$ the moduli space of Floer trajectories
\begin{equation} \label{eq:Floer}
\begin{aligned}
 & u: \R \times [0,1] \longrightarrow M, \\
 & u(s,0) \in L_0, \;\; u(s,1) \in L_1, \\
 & \partial_s u + J(t,u) \big(\partial_tu - X(t,u)\big) = 0, \\
 & \textstyle\lim_{s \rightarrow \pm\infty} u(s,t) = y_\pm(t),
\end{aligned}
\end{equation}
where $X$ denotes the Hamiltonian vector field of $H$, meaning that $\omega(X,\cdot) = dH$, and $y_\pm$ are flow lines of $X$ with $y_\pm(1) = x_\pm$. To further clarify the sign conventions used here, we remind the reader that \eqref{eq:Floer} can be viewed as the gradient flow equation for the action functional on path space $\scrP = \{y: [0,1] \rightarrow M, \; y(0) \in L_0, \; y(1) \in L_1\}$. If we write $\theta|L_k = dF_k$, then this is given by
\begin{equation}
A(y) = \Big(\int_0^1 y^*\theta + H_t(y(t))\, dt\Big) + F_0(y(0)) - F_1(y(1)).
\end{equation}

\subsection{The equivariant theory\label{Section:equivariant}}
We now add symmetry, in the form of an involution $\iota$ of $M$ preserving $\o$, $\theta$ and $I$. We also assume that both Lagrangian submanifolds are invariant under this symmetry. One then has an equivariant Floer cohomology $HF_{\mathrm{borel}}(L_0,L_1)$, which is a module over $\Z_2[[q]]$. To define it, we return to the Borel construction \eqref{eq:borel-constr}. Each fibre $M_{\mathrm{borel},z}$ contains canonical Lagrangian submanifolds $L_{0,z}$ and $L_{1,z}$. In addition, choose a time-dependent Hamiltonian and almost complex structure $(H_z,J_z)$, smoothly depending on $z$. When $z = z^{(k)}$ is a critical point, the $H^{(k)}$ should satisfy the transverse intersection condition required to make the associated Floer cochain space $CF(L_0,L_1)^{(k)}$ well-defined. Moreover, the family $(H_z,J_z)$ should be locally constant near each critical point. As in \eqref{eq:borel-complex} we then set
\begin{equation} \label{eq:borel-floer}
CF_{\mathrm{borel}}(L_0,L_1) = \prod_{k=0}^{\infty} CF(L_0,L_1)^{(k)}.
\end{equation}
The differential $d_{\mathrm{borel}}$ counts solutions of the coupled equations
\begin{equation} \label{eq:coupled}
\begin{aligned}
& v: \R \longrightarrow \R P^{\infty}, \\
& dv/ds = \nabla h(v), \\
& u: \R \times [0,1] \longrightarrow M_{\mathrm{borel}}, \\
& (\partial_s u)^{\mathrm{vert}} + J_{v(s)}(t,u) \Big((\partial_t u)^{\mathrm{vert}} - X_{v(s)}(t,u)\Big) = 0
\end{aligned}
\end{equation}
with asymptotics $(z^{(j)},x^{(j)}_-)$ and $(z^{(k)},x^{(k)}_+)$. This is analytically unproblematic, since \eqref{eq:coupled} is either Floer's equation (if $v$ is constant), or else a continuation map equation for $u$ parametrized by the finite-dimensional moduli space of gradient flow lines $v$. For generic choices of almost complex structures, the resulting moduli spaces $\scrM_{\mathrm{borel}}(x^{(j)}_-,x^{(k)}_+)$ are smooth and admit a compactification with the same properties as in the Morse-theoretic case. Note that in general, the energy bounds depend on $j$ and $k$, which means that the differential can contain infinitely many terms. Hence, the use of the direct product in \eqref{eq:borel-floer} is mandatory.
%By construction, $CF_{\mathrm{borel}}(L_0,L_1)$ admits a decreasing filtration by $k$, which is %complete and compatible with the differential. In parallel with \eqref{eq:borel-ss}, this yields a %spectral sequence
%\begin{equation} \label{eq:borel-floer-ss}
%HF(L_0,L_1)[[q]] \Longrightarrow HF_{\mathrm{borel}}(L_0,L_1),
%\end{equation}
%which is particularly useful in comparison arguments. For instance, if $HF(L_0,L_1)$ vanishes, then so %does its equivariant version. The same step appears when proving that equivariant Floer cohomology is %invariant under (exact, equivariant, compactly supported) Lagrangian isotopies.

It is possible to simplify the setup by requiring compatibility of our choices with the self-embedding $\tau: \R P^\infty \rightarrow \R P^\infty$. In that situation
\begin{equation}
\begin{aligned}
& CF_{\mathrm{borel}}(L_0,L_1) = CF(L_0,L_1)[[q]],
& d_{\mathrm{borel}} = d^{(0)} + d^{(1)}q + d^{(2)}q^2 + \cdots
\end{aligned}
\end{equation}
The leading term $d^{(0)}$ is the ordinary Floer differential associated to $(H,J) = (H^{(0)},J^{(0)})$. The next term $d^{(1)}$ measures the difference between two continuation maps. One is a continuation for a family connecting $(H,J)$ to its pullback $(\iota^*H,\iota^*J)$, and the other connects $(H,J)$ to itself. The latter can be chosen constant, in which case the situation is precisely analogous to \eqref{eq:d-1}.
%In particular, again imitating the finite-dimensional behaviour, we find that the first differential %in \eqref{eq:borel-floer-ss} is $\iota_* - \mathrm{id}$, where $\iota_*$ is the involution on %$HF(L_0,L_1)$ induced by $\iota$.

It was proved in \cite[Proposition 5.13]{khovanov-seidel} that for a generic invariant choice of $(H,J)$, all solutions of \eqref{eq:Floer} which are not entirely contained in $M^{\mathrm{inv}}$ will be regular. In particularly simple cases, this already enables one to get a better understanding of equivariant Floer cohomology. One such case is the following:

\begin{lemma}
If $L_0 \cap L_1 \cap M^{\mathrm{inv}} = \emptyset$, then $HF_{\mathrm{borel}}(L_0,L_1)$ is of finite rank over $\Z_2$.
\end{lemma}

\begin{proof}
In this case, we can clearly choose $(H,J)$ to be invariant and regular. This allows us to copy the strategy from Lemma \ref{th:acyclic}, which means to find an acyclic subspace $Z \subset CF_{\mathrm{borel}}(L_0,L_1)$ such that $CF_{\mathrm{borel}}(L_0,L_1)/Z$ has exactly one generator for each orbit of intersection points.
\end{proof}

In general, equivariant transversality fails to hold, and not just for technical reasons. There are topological obstructions that have no parallel in finite-dimensional Morse theory, and which affect the overall picture substantially. This phenomenon had already been noticed in the context of monopole Floer homology \cite{KM}, which can be viewed as an $S^1$-equivariant theory.

\subsection{An example\label{subsec:moebius}}
Take $M = \R \times S^1 \times \R^2$ with coordinates $(p_1,q_1,p_2,q_2)$, where $q_1 \in S^1 = \R/\Z$, and the standard symplectic structure $\omega = dp_1 \wedge dq_1 + dp_2 \wedge dq_2$. This is invariant under the involution $\iota(p_1,q_1,p_2,q_2) = (p_1,q_1,-p_2,-q_2)$. One of our Lagrangian submanifolds will be the cylinder
\begin{equation}
L_0 = \{p_1 = p_2 = 0\}.
\end{equation}
The other is a Moebius band, constructed from two overlapping pieces. Fix some small $\delta>0$. Choose a function $\psi_+: (-\delta,1/2+\delta) \rightarrow \R$ such that $\psi_+(s) = -1$ for $s \leq \delta$; $\psi_+(s) = 1$ for $s \geq 1/2-\delta$; and $\psi_+'(s) > 0$ for $s \in (\delta,1/2-\delta)$. Similarly, take $\psi_-: (1/2-\delta,1+\delta) \rightarrow \R$ such that $\psi_-(s) = -1$ for $s \leq 1/2+\delta$; and $\psi_-(s) = 1$ for $s \geq 1-\delta$. Then
\begin{equation} \label{eq:l0}
\begin{aligned}
& L_1 = L_{1,+} \cup L_{1,-}, \\
& L_{1,+} = \{ q_1 \in (-\delta,1/2+\delta), \; p_1 = \psi'_+(q_1) q_2^2/2, \; p_2 = \psi_+(q_1) q_2 \}, \\
& L_{1,-} = \{ q_1 \in (1/2-\delta,1+\delta), \; p_1 = \psi'_-(q_1) p_2^2/2, \; q_2 = -\psi_-(q_1) p_2 \}.
\end{aligned}
\end{equation}
If we write $\omega = d\theta$ with $\theta = p_1 \, dq_1 + (p_2 \, dq_2 - q_2 \, dp_2)/2$, both $L_k$ are exact. Moreover, the associated Liouville flow $\rho_t(p_1,q_1,p_2,q_2) = (e^t p_1, q_1, e^{t/2} p_2, e^{t/2} q_2)$ preserves the $L_k$. Hence, if we choose a suitable almost complex structure $I$ (not the standard complex structure, but one which is compatible with the Liouville flow at infinity), maximum principle arguments (compare \cite[Lemma 5.5]{khovanov-seidel}, for instance) ensure that $\mathit{HF}^*(L_0,L_1)$ is well-defined. The same holds for their fixed parts. The following computation shows that the Smith inequality is violated in this case:

\begin{lemma}
$\mathit{HF}^*(L_0^{\mathrm{inv}},L_1^{\mathrm{inv}}) \iso \Z_2 \oplus \Z_2$, whereas $\mathit{HF}^*(L_0,L_1) = 0$.
\end{lemma}

\begin{proof}
The first part is easy: inside $M^{\mathrm{inv}} = \R \times S^1$, we have $L_0^{\mathrm{inv}} = L_1^{\mathrm{inv}} = \{0\} \times S^1$, hence $\mathit{HF}^*(L_0^{\mathrm{inv}},L_1^{\mathrm{inv}}) \iso H^*(S^1;\Z_2)$.

For the second part, we need to take a closer look at the construction \eqref{eq:l0}. Identify the subset $\R \times (-\delta,1/2+\delta) \times \R^2 \subset M$ with the cotangent bundle of $(-\delta,1/2+\delta) \times \R$, where the base coordinates are $(q_1,q_2)$ and the fibre coordinates $(p_1,p_2)$. In this picture $L_0$ is the zero-section, and $L_1$ is the graph of the derivative of the function $f_+(q_1,q_2) = \psi_+(q_1)q_2^2/2$. Hence, intersection points of these two submanifolds correspond to critical points of $f_+$, which are exactly the points where $q_2 = 0$. Exactly one of those, corresponding to the unique solution of $\psi_+(q_1) = 0$, is not a Morse-Bott critical point (which means that $L_0$ and $L_1$ fail to intersect cleanly there). Now identify $\R \times (1/2-\delta,1+\delta) \times \R^2 \subset M$ with the cotangent bundle of $(1/2-\delta,1+\delta) \times \R$, but where this time the base coordinates are $(q_1,p_2)$ and the fibre coordinates $(p_1,-q_2)$. There, $L_0$ is the conormal bundle of $(1/2-\delta,1+\delta) \times \{0\}$, and $L_1$ is the graph of the derivative of $f_-(q_1,p_2) = \psi_-(q_1)p_2^2/2$. Intersection points correspond to critical points of $f_-|(1/2-\delta,1+\delta) \times \{0\}$, which are all Morse-Bott since that function vanishes identically.

Suppose that we deform $f_\pm$ to $\tilde{f}_\pm = f_\pm + \epsilon r(q_1)$, where $r$ is a Morse function on $S^1$ with a minimum at $0$ and a maximum at $1/2$, and $\epsilon>0$ small. In the first cotangent bundle picture above, $\tilde{f}_+$ has only the two obvious critical points $(0,0)$ and $(1/2,0)$, for any sufficiently small value of $\epsilon$, which means that there are no additional critical points branching out from the degenerate one. This is by a sign consideration: $\partial \tilde{f}_+/\partial q_1 = \psi'(q_1)q_2^2/2 + r'(q_1) > 0$ whenever $q_1 \in (0,1/2)$. In the second cotangent bundle picture, the same consequence holds for more standard reasons (this is a Morsification of a clean intersection, as in \cite{pozniak}). We actually want to deform $L_1$ by an exact isotopy which is fixed at infinity and which, near its intersection with $L_0$, corresponds to the change of Morse functions discussed above. Denote the result by $\tilde{L}_1$.

The Floer complex $CF^*(L_0,\tilde{L}_1)$ has two generators $x_\pm$, corresponding to the minimum and maximum of $r$. Choose $(J_t)$ to be $\iota$-invariant but otherwise generic. The moduli space $\scrM(x_-,x_+)$ is divided into open and closed subsets corresponding to different homotopy classes, and each of these subsets has a different expected dimension. The subspace $\scrM(x_-,x_+)^{\mathrm{inv}}$ contains exactly two solutions, clearly visible as strips in $M^{\mathrm{inv}} = \R \times S^1$. Consider the strip which lies in $\R \times (1/2,1)$. As part of the standard comparison between Floer theory and Morse theory, the index of this strip can be computed as the difference of the Morse indices of $\tilde{f}_+|\{p_2 = 0\}$ at the endpoints. As a consequence, this particular strip lies in the part of $\scrM(x_-,x_+)$ having expected dimension $0$ (the other strip, inside $\R \times (0,1)$, turns out to have expected dimension $-1$). Hence, if we choose $(J_t)$ generic (while preserving its symmetry), this strip will be regular in the total space, and in fact the zero-dimensional part of $\scrM(x_-,x_+)$ will be regular and contain an odd number of points. The parts of $\scrM(x_-,x_+)$ of different expected dimension may not be regular, but they will become regular if we make a small perturbation which breaks the symmetry. The Floer differential, computed after making such a perturbation, is therefore nonzero.
\end{proof}

\subsection{The normal polarization}
The key to the example above was the difference between Maslov indices in $M$ and $M^{\mathrm{inv}}$. However, this is only the simplest in a series of obstructions, and the remaining ones have nothing to do with the grading on Floer cohomology. The following general discussion of these obstructions is only for motivation, and (except for some notation which will reappear elsewhere) not needed for the main results of the paper.

Let $L_k^{\mathrm{inv}} = L_k \cap M^{\mathrm{inv}}$ be the part of $L_k$ fixed by the involution. Take the fixed part of the path space, $\scrP^{\mathrm{inv}} = \{y: [0,1] \rightarrow M^{\mathrm{inv}}, \; y(0) \in L_0^{\mathrm{inv}}, \; y(1) \in L_0^{\mathrm{inv}}\}$. The normal polarization is a map
\begin{equation} \label{eq:lgrass}
\scrP^{\mathrm{inv}} \longrightarrow U/O,
\end{equation}
unique up to homotopy, which describes the geometry in normal direction. It can be defined as follows. Let $TM^{\anti} \longrightarrow [0,1] \times M^{\mathrm{inv}}$ be the normal bundle to the fixed locus, pulled back to $[0,1] \times M^{\mathrm{inv}}$. On the subsets $\{k\} \times L_k^{\mathrm{inv}}$ this carries canonical Lagrangian subbundles $TL_k^{\mathrm{anti}}$, given by the normal bundles to the fixed parts of the Lagrangian submanifolds. This bundle-theoretic datum is classified by a map
\begin{equation} \label{eq:pair-classifying}
\big([0,1] \times M^{\mathrm{inv}}, (\{0\} \times L_0^{\mathrm{inv}}) \cup (\{1\} \times L_1^{\mathrm{inv}})\big) \longrightarrow (BU,BO).
\end{equation}
By looking at the graph of any path, this induces a map \eqref{eq:lgrass}. Note that we can pass to loop spaces one more time, and (using real Bott periodicity) obtain a map
\begin{equation} \label{eq:double-loop}
\Omega \scrP^{\mathrm{inv}} \longrightarrow \Omega(U/O) \htp \Z \times BO.
\end{equation}
There is an equivalent infinite-dimensional viewpoint, related to the previous one through index theory. For each $y \in \scrP^{\mathrm{inv}}$, let $\scrH_y$ be the Hilbert space of $W^{1/2}$-sections of $y^*TM^{\mathrm{anti}}$ with boundary conditions $y(k) \in TL_k^{\mathrm{anti}}$ (since $W^{1/2}$-functions are not continuous, this should strictly speaking be defined as the closure of the space of smooth functions with the same boundary conditions). $\scrH_y$ carries a bounded symmetric bilinear form
\begin{equation}
Q_y(Y_1,Y_2) = \omega\big( Y_1, \frac{\nabla Y_2}{dt} \big)
\end{equation}
where $\nabla$ is any symplectic connection on $TM^{\mathrm{anti}}$. This leads to a polarization of the Hilbert bundle $\scrH \rightarrow \scrP^{\mathrm{inv}}$ in the sense of \cite{CJS}. Like any polarized Hilbert bundle, this is classified by a map $\scrP^{\mathrm{inv}} \rightarrow BGL_{\mathrm{res}}$ into the restricted Grassmannian of a real Hilbert space. The connection with the previous viewpoint is given by \cite[Proposition 6.2.4]{PS}, which says that there is a homotopy equivalence $BGL_{\mathrm{res}} \htp U/O$.
%In essentially equivalent
%
%Similarly, $GL_{\mathrm{res}}$ itself is homotopy equivalent to the space of real Fredholm operators %\cite{atiyah-singer}, which in turn is homotopy equivalent to $\Z \times BO$, and in this way one %recovers \eqref{eq:double-loop}.
%\end{remark}

To see the relevance of this to transversality problems, consider \eqref{eq:Floer} for generic invariant $(H,J)$. Let $u$ be a solution which is entirely contained in $M^{\mathrm{inv}}$, with limits $x_\pm$. The associated linearized operator $D_u$ is itself equivariant, hence can be decomposed into its invariant and anti-invariant parts
\begin{equation}
D_u = D_u^{\mathrm{inv}} \oplus D_u^{\mathrm{anti}}.
\end{equation}
$D_u^{\mathrm{inv}}$ controls deformations which remain inside $M^{\mathrm{inv}}$. By the standard transversality arguments, we can assume that this is onto for all $u$. The anti-invariant part governs deformations in normal direction. For equivariant transversality, it is necessary that these operators be onto as well, and our definition of localization map further requires them to be actually invertible. There is a topological obstruction to this, given by an appropriate family Fredholm index. To express this in slightly simpler terms, fix a base point $u_0 \in \scrM(x_-,x_+)^{\mathrm{inv}}$, and consider the family of real Fredholm operators $D_u \oplus D_{u_0}^*$ over $\scrM(x_-,x_+)^{\mathrm{inv}}$. On the other hand, we have a map $\scrM(x_-,x_+)^{\mathrm{inv}} \rightarrow \Omega \scrP^{\mathrm{inv}}$, unique up to homotopy, which is constructed by gluing together $u$ and $u_0$ topologically to an annulus with boundary on $(L_0^{\mathrm{inv}},L_1^{\mathrm{inv}})$. The family index theorem shows that our family of operators $D_u \oplus D_{u_0}^*$ is classified by the composition of this map with \eqref{eq:double-loop}.

In the example from Section \ref{subsec:moebius}, the bundle $TM^{\mathrm{anti}}$ is trivial. If we choose the obvious trivialization, then $TL_0^{\mathrm{anti}}$ is compatible with that, but $TL_1^{\mathrm{anti}}$ is not, being a family of Lagrangian subspaces over $S^1$ with nonzero Maslov index. This means that the relative first Chern class $2c_1^{\mathrm{rel}} \in H^2(BU,BO;\Z)$ evaluates nontrivially on \eqref{eq:pair-classifying}. As a consequence, the map $\pi_0(\Omega \scrP^{\mathrm{inv}}) \rightarrow \pi_0(\Z \times BO) = \Z$ is nontrivial. This shows that holomorphic strips $u$ in $M^{\mathrm{inv}}$ with the same endpoints can have operators $D_u^{\mathrm{anti}}$ with different indices, which is indeed what we saw in our previous concrete computation.

\subsection{Equivariant transversality\label{subsec:stable-normal}}
While \eqref{eq:lgrass} appears to be the fundamental obstruction, we will assume a stronger condition which works on the level of \eqref{eq:pair-classifying}, and essentially amounts to a nullhomotopy of that map. We consider $M$, $(L_0,L_1)$, and $\iota$ as in Section \ref{subsec:borel-floer}. Let $n$ be the complex dimension of $M$, and $n_{\mathrm{anti}}$ the complex codimension of the fixed point set\footnote{The case where the fixed point set has components of different dimensions would be an easy generalization.}.

\begin{definition} \label{definition:stablenormal}
A stable normal trivialization consists of the following data:
\begin{itemize}
\item A trivialization of unitary vector bundles over $M^{\mathrm{inv}}$,
\begin{equation}
\phi: TM^{\mathrm{anti}} \oplus \C^N \longrightarrow \C^{n_{\mathrm{anti}} + N}
\end{equation}
for some $N \geq 0$.
\item
A Lagrangian subbundle $\Lambda_0 \subset (TM^{\mathrm{anti}} \oplus \C^N) | ([0,1] \times L_0)$ such that $\Lambda_0 | \{0\} \times L_0 = TL_0^{\mathrm{anti}} \oplus \R^N$ and $\phi(\Lambda_0 | \{1\} \times L_0) = \R^{n_{\mathrm{anti}}+N}$.
\item
A Lagrangian subbundle $\Lambda_1 \subset (TM^{\mathrm{anti}} \oplus \C^N) | ([0,1] \times L_1)$ such that $\Lambda_1 | \{0\} \times L_1 = TL_1^{\mathrm{anti}} \oplus i\R^N$ and $\phi(\Lambda_1 | \{1\} \times L_1) = i\R^{n_{\mathrm{anti}} + N}$.
\end{itemize}
If this exists, we say that $(M,L_0,L_1)$ have stably trivial normal structure.
\end{definition}

Assume from now on that we are given a stable normal trivialization. Take the product $\tilde{M} = M \times \C^N$ with the standard product symplectic structure. Equip it with the involution $\tilde{\iota} = \iota \times (-\mathrm{id})$, so that $\tilde{M}^{\mathrm{inv}} = M^{\mathrm{inv}}$. Fix a relatively compact open subset $U^{\mathrm{inv}} \subset M^{\mathrm{inv}}$. By a standard Moser argument, there is an open subset $\tilde{U} \subset \tilde{M}$ with $\tilde{U} \cap \tilde{M}^{\mathrm{inv}} \supset U^{\mathrm{inv}}$ and a (codimension zero) symplectic embedding
\begin{equation} \label{eq:tilde-phi}
\tilde\Phi: \tilde{U} \longrightarrow M^{\mathrm{inv}} \times \C^{n_{\mathrm{anti}} + N}
\end{equation}
which is $G$-equivariant, equals the identity on the fixed point set, and whose normal derivative along the fixed point set is given by the trivialization $\phi$. Similarly, there are $G$-invariant Lagrangian submanifolds $\tilde{L}_k \subset \tilde{M}$, $k = 0,1$, with the following properties. $\tilde{L}_0$ is obtained from $L_0 \times \R^N$ by a compactly supported exact Lagrangian isotopy, and similarly $\tilde{L}_1$ from $L_1 \times i\R^N$. These isotopies are trivial on the fixed part, and
\begin{equation} \label{eq:local-picture}
\begin{aligned}
& \tilde{L}_0 \cap \tilde{U} = \tilde{\Phi}^{-1}(L^{\mathrm{inv}}_0 \times \R^{n_{\mathrm{anti}} + N}), \\
& \tilde{L}_1 \cap \tilde{U} = \tilde{\Phi}^{-1}(L^{\mathrm{inv}}_1 \times i\R^{n_{\mathrm{anti}} + N}). \end{aligned}
\end{equation}
To define the relevant isotopies, one proceeds in two steps. The infinitesimal change of $\Lambda_k$ in direction of the $[0,1]$ variable is expressed by a fibrewise quadratic form. Extend this arbitrarily to a real quadratic form on $TM^{\mathrm{anti}} \oplus \C^N \rightarrow [0,1] \times M^{\mathrm{inv}}$. Next, find a compactly supported time-dependent Hamiltonian function $H_k \in \smooth([0,1] \times \tilde{M},\R)$ which is $G$-invariant, vanishes on the fixed point set, and whose Hessians along $U$ are given by the previously constructed quadratic form. The associated Hamiltonian isotopy then produces Lagrangian submanifolds which satisfy \eqref{eq:local-picture} infinitesimally, meaning that their derivatives in direction normal to the fixed point sets map to $\R^{n_{\mathrm{anti}} + N}$ and $i\R^{n_{\mathrm{anti}} + N}$ under $\phi$. Improving this to the desired statement is another standard Moser type construction.

For the application to equivariant transversality problems, we fix a time-dependent Hamiltonian and almost complex structure $(H^{\mathrm{inv}},J^{\mathrm{inv}})$ on $M^{\mathrm{inv}}$ which are suitable for defining the Floer cohomology inside the fixed point set. This requirement includes the necessary regularity property of all moduli spaces $\scrM(x_-,x_+)^{\mathrm{inv}}$. We choose $U^{\mathrm{inv}}$ so that for any map $u$ in these moduli spaces, $u(\R \times [0,1]) \subset U^{\mathrm{inv}}$, which is possible due to the convexity assumptions we have imposed. Now, consider data $(\tilde{H},\tilde{J})$ on $\tilde{M}$ which are $G$-equivariant and, in addition to the usual properties, satisfy the following conditions:
\begin{itemize}
\item $\tilde{H}|U^{\mathrm{inv}} = H^{\mathrm{inv}}$, and the Hessian of any $\tilde{H}(t,\cdot)$ in normal direction to the fixed point set vanishes along $U^{\mathrm{inv}}$.
\item Along $U^{\mathrm{inv}}$, $\tilde{J}$ agrees with the pullback of the product structure $J^{\mathrm{inv}} \times i$ by \eqref{eq:tilde-phi}. Moreover, the first derivatives of these two structures also agree along $U^{\mathrm{inv}}$.
\end{itemize}
Take any $u \in \scrM(x_-,x_+)^{\mathrm{inv}}$, consider it as a solution of Floer equation in $\tilde{M}$ associated to $(\tilde{H},\tilde{J})$, and take the associated linearized operator $\tilde{D}_u$. By assumption, its invariant part $\tilde{D}_u^{\mathrm{inv}}$ is surjective. The conditions imposed above ensure that the anti-invariant part $\tilde{D}_u^{\mathrm{anti}}$ is the standard $\bar\partial$-operator for maps $[0,1] \times \R \longrightarrow \C^{n_{\mathrm{anti}} + N}$ with real and imaginary boundary conditions along $\{0\} \times \R$ and $\{1\} \times \R$, respectively; which is invertible.

\begin{Lemma} \label{lem:generic-regular}
For a generic choice of $(\tilde{H},\tilde{J})$ subject to the conditions above, all solutions of \eqref{eq:Floer} are regular.
\end{Lemma}

\begin{proof}
This is essentially the same argument as in \cite[Section 14c]{FCPLT}. By construction, all invariant solutions are regular. The others are taken care of by \cite[Proposition 5.13]{khovanov-seidel}; the additional requirements on $(\tilde{H},\tilde{J})$ are unproblematic here, since they do not constrain the behaviour outside the fixed point set.
\end{proof}

\subsection{The localization map} \label{subsec:Floerlocalize}
It is not hard to see that passing from $M$ and $(L_0,L_1)$ to their stabilized versions $\tilde{M}$ and $(\tilde{L}_0,\tilde{L}_1)$ does not affect Floer cohomology or its equivariant version. In fact, if one chooses the almost complex structures and Hamiltonian perturbations in a specific way, the chain complexes will be the same. For generally unrelated choices, this means that we have quasi-isomorphisms
\begin{equation} \label{eq:q1}
\begin{aligned}
& CF(L_0,L_1) \xrightarrow{\htp} CF(\tilde{L}_0,\tilde{L}_1), \\
& CF_{\mathrm{borel}}(L_0,L_1) \xrightarrow{\htp} CF_{\mathrm{borel}}(\tilde{L}_0,\tilde{L}_1),
\end{aligned}
\end{equation}
which are themselves unique up to homotopy. From now on, we fix a stable normal trivialization, and define $(\tilde{L}_0,\tilde{L}_1)$ as well as $(\tilde{H},\tilde{J})$ as in the previous section. This in particular allows us to define a complex $CF_{\mathrm{equiv}}(\tilde{L}_0,\tilde{L}_1)$ as in Section \ref{sec:invariant-morse}, which comes with a canonical quasi-isomorphism
\begin{equation} \label{eq:q2}
CF_{\mathrm{borel}}(\tilde{L}_0,\tilde{L}_1) \xrightarrow{\htp} CF_{\mathrm{equiv}}(\tilde{L}_0,\tilde{L}_1).
\end{equation}
The proof that this is a quasi-isomorphism, carried over from Lemma \ref{th:acyclic}, uses the nilpotency of the operator $U$, which holds because we have a well-defined action functional underlying our Morse theories.

Write $\scrM(x_-,x_+)$ for the moduli spaces associated to $(\tilde{H},\tilde{J})$. Invertibility of the anti-invariant part $\tilde{D}_u$ for any $u \in \scrM(x_-,x_+)^{\mathrm{inv}}$ means that $\scrM(x_-,x_+)^{\mathrm{inv}} \subset \scrM(x_-,x_+)$ is an open and closed subset, and the same holds for the compactifications. Following the process from Section \ref{sec:invariant-morse}, one uses these moduli spaces to define a bare localization chain map
\begin{equation}
\lambda: CF_{\mathrm{equiv}}(\tilde{L}_0,\tilde{L}_1) \longrightarrow
CF(\tilde{L}_0^{\mathrm{inv}},\tilde{L}_1^{\mathrm{inv}})[[q]] =
CF(L_0^{\mathrm{inv}},L_1^{\mathrm{inv}})[[q]].
\end{equation}
From a technical viewpoint, this goes either via a choice of consistent sections, or alternatively via singular chains as outlined in Remark \ref{rem:cocycle}. It is useful to keep in mind that, because of exactness, the overall dimension of our moduli spaces is bounded above, which means that the formal power series in $\partial_q$ entering into the definition are still polynomials. Moreover, the analogues of Lemmas \ref{th:rescale-map} and \ref{th:localization} hold. The outcome, combined with our previous observations \eqref{eq:q1} and \eqref{eq:q2}, is the following:

\begin{Theorem} \label{th:th}
There is a sequence of $\Z_2[[q]]$-module maps
\begin{equation}
\Lambda^{(m)}: HF_{\mathrm{borel}}(L_0,L_1) \longrightarrow HF(L_0^{\mathrm{fix}},L_1^{\mathrm{fix}})[[q]].
\end{equation}
defined for $m \gg 0$ and satisfying $\Lambda^{(m+1)} = q\Lambda^{(m)}$. Moreover, these maps become isomorphisms after tensoring with $\Z_2((q))$.
\end{Theorem}

One can prove that these maps depend only on the (homotopy class of the) stable normal trivialization. We will not give the argument here. By construction, equivariant Floer cohomology fits into a long exact sequence analogous to \eqref{eq:uct},
\begin{equation} \label{eqn:uct-floer}
\cdots \rightarrow HF_{\mathrm{borel}}(L_0,L_1) \xrightarrow{q} HF_{\mathrm{borel}}(L_0,L_1) \rightarrow HF(L_0,L_1) \rightarrow \cdots
\end{equation}
By the same argument as in Section \ref{sec:smith-inequality}, this and Theorem \ref{th:th} imply the Floer-theoretic Smith inequality, in the form stated in the Introduction (Theorem \ref{thm:mainsmith}).

A particularly simple numerical consequence is as follows. Assume that $L_0^{\mathrm{inv}}$ and $L_1^{\mathrm{inv}}$ are orientable. Then, $HF(L_0^{\mathrm{inv}},L_1^{\mathrm{inv}})$ admits a modulo $2$ grading so that its Euler characteristic is, up to a dimension-dependent sign, the intersection number $\pm [L_0^{\mathrm{inv}}] \cdot [L_1^{\mathrm{inv}}]$. One can refine this by taking the fundamental groups into account. Namely, for every connected component $C$ of the path space $\scrP^{\mathrm{inv}}$ we have a direct summand $HF_C(L_0^{\mathrm{inv}},L_1^{\mathrm{inv}})$, whose Euler characteristic $\pm [L_0^{\mathrm{inv}}] \cdot_C [L_1^{\mathrm{inv}}]$ takes into account only the intersection points whose associated constant paths lie in $C$. The Smith inequality immediately yields the following lower bound:

\begin{corollary} \label{cor:separate-components}
If in addition to our other assumptions, $L_0^{\mathrm{inv}}$ and $L_1^{\mathrm{inv}}$ are orientable, we have
\begin{equation}
\mathrm{dim} \, HF(L_0,L_1) \geq \sum_C \Big| [L_0^{\mathrm{inv}}] \cdot_C [L_1^{\mathrm{inv}}] \Big|.
\end{equation}
\end{corollary}

\begin{remark} \label{th:twist}
Let's return to the general situation from Section \ref{subsec:borel-floer}, without the assumption of stable normal triviality. We can still choose $(H,J)$ equivariantly so that their fixed parts $(H^{\mathrm{inv}},J^{\mathrm{inv}})$ give rise to regular moduli spaces inside $M^{\mathrm{inv}}$. Suppose that the compactified moduli spaces $\bar\scrM^{\mathrm{inv}} = \bar\scrM(x_-,x_+)^{\mathrm{inv}}$ have been equipped with fundamental chains, as in Remark \ref{rem:cocycle}. In addition, they carry canonical families of Fredholm operators $D^{\mathrm{anti}}_u$ measuring the obstruction to normal regularity, in the sense of \eqref{eq:double-loop}. Denote by
\begin{equation}
\eta \longrightarrow \bar\scrM^{\mathrm{inv}}
\end{equation}
the negative virtual index bundle of the family (negative means taking $[\mathrm{cokernel}] - [\mathrm{kernel}]$). The restriction of this bundle to any boundary stratum in the compactification naturally decomposes as a direct sum of the bundles associated to the various factors (this is just an abstract formulation of the usual gluing theorem for indices). Suppose that we are given singular cocycle representatives $p$ of the total Stiefel-Whitney polynomial,
\begin{equation}
[p] = 1 + w_1(\eta) + w_2(\eta) + \cdots \in H^*(\bar\scrM^{\mathrm{inv}};\Z_2),
\end{equation}
which moreover should be compatible with each other via restriction to the boundary strata and the isomorphisms mentioned above (we will not consider the technical details of how to construct such representatives here). Now pass to cochains with coefficients in $\Z_2((q))$, and make $p$ into a homogeneous element $\hat{p}$ of degree $-\mathrm{index}(D^{\mathrm{anti}}_u)$, by multiplying its graded pieces with appropriate powers of $q^{\pm 1}$ (if the moduli space has components with different indices, this has to be done separately on each component, of course). One can then define a twisted Floer differential on $CF_{\mathrm{twisted}}(L_0^{\mathrm{inv}},L_1^{\mathrm{inv}}) = CF(L_0^{\mathrm{inv}},L_1^{\mathrm{inv}})((q))$ by evaluating these cocycles on the fundamental chains:
\begin{equation}
d_{\mathrm{twisted}}(x_-) = \sum_{x_+} \langle \hat{p}, [\bar\scrM(x_-,x_+)^{\mathrm{inv}}] \rangle x_+.
\end{equation}
The resulting cohomology groups $HF_{\mathrm{twisted}}(L_0^{\mathrm{inv}},L_1^{\mathrm{inv}})$ reduce to ordinary Floer cohomology in the presence of stable normal trivialization. We would like to propose these twisted groups as the candidate target for a localization map defined in general.

To illustrate the effect of twisting, we can return to the example from Section \ref{subsec:moebius}. In that case, the unique nontrivial moduli space $\scrM(x_-,x_+)^{\mathrm{inv}}$ consisted of two points, one of which was regular in $M$ (which means that $D^{\mathrm{anti}}_u$ has index $0$), whereas the other was irregular (the index would be $-1$). By definition, $\hat{p}$ is $1$ on one point and $q$ on the other point, so that
\begin{equation}
d_{\mathrm{twisted}}(x_-) = (1+q)x_+.
\end{equation}
Therefore $HF_{\mathrm{twisted}}(L_0^{\mathrm{inv}},L_1^{\mathrm{inv}}) = 0$, which is consistent with the conjectured existence of a localization map.
\end{remark}

\section{From symplectic Khovanov to Heegaard Floer\label{sec:four}}

\subsection{Algebraic geometry and topology\label{sec:equivariant}}
We begin by recalling some features of the specific transverse slices from \cite{SS}, and simultaneously equip them with involutions. Fix an integer $m \geq 1$. Let $\Slice \subset \mathfrak{sl}_{2m}(\C)$ be the affine subspace consisting of matrices of the form
\begin{equation} \label{eq:y-matrix}
A = \begin{pmatrix}
A_{1} & I &&& \\
A_{2} && I && \\
\dots &&& \dots & \\
A_{m-1} &&&& I \\
A_{m} &&&& 0
\end{pmatrix}
\end{equation}
with $A_1 \in \mathfrak{sl}_2(\C)$, $A_k \in \mathfrak{gl}_2(\C)$ for $k>1$, and where $I \in \mathfrak{gl}_2(\C)$ is the identity matrix. Let $\Sym^0_{2m}(\C)$ be the subspace of the symmetric product $\Sym_{2m}(\C)$ consisting of collections with center of mass zero. Symmetric polynomials yield an isomorphism $\Sym^0_{2m}(\C) \iso \C^{2m-1}$. Consider the adjoint quotient map $\chi: \Slice \rightarrow \Sym^0_{2m}(\C)$, which takes a matrix $A$ to the collection of its eigenvalues. If we identify $\Sym^0_{2m}(\C) \iso \C^{2m-1}$ as before, this map is just given by the nontrivial coefficients of the characteristic polynomial. In our case,
\begin{equation} \label{eq:det2}
\begin{aligned}
& \det(x - A) = \det(A(x)), \quad \text{where} \\
& A(x) = x^m I - x^{m-1} A_1 - x^{m-2} A_2 - \cdots - A_m \in \mathfrak{gl}_2(\C[x]).
\end{aligned}
\end{equation}
The part of $\chi$ lying over the open subset $\Conf_{2m}^0(\C) \subset \Sym^0_{2m}(\C)$ of configurations (unordered $2m$-tuples of pairwise distinct points) is a differentiable fibre bundle. Fix some $t \in \Conf_{2m}^0(\C)$, and denote the fibre of $\chi$ at that point by $\Y$. By definition, this is a smooth affine variety of complex dimension $2m$. Consider the holomorphic involution which transposes each of the $A_k$:
\begin{equation}\label{eq:involution}
\iota : \Slice \longrightarrow \Slice, \qquad \iota(A_1,\ldots,A_m) = (A_1^{\mathrm{tr}},\ldots,A_m^{\mathrm{tr}}).
\end{equation}
Clearly, the function \eqref{eq:det2} is $\iota$-invariant, hence we get an induced involution on $\Y$. The fixed point sets $\Slice^{\mathrm{inv}} \subset \Slice$ and $\Y^{\mathrm{inv}} \subset \Y$ are defined by the vanishing of the $m$ functions $(A_i)_{12}-(A_i)_{21}$.

Manolescu \cite[Theorem 1.1]{Manolescu} showed that $\Y$ can be identified with an open subset of a Hilbert scheme. More precisely, for our given $t \in \Conf_{2m}^0(\C)$ and the associated monic polynomial $f(x) = \prod_i (x-t_i)$, consider the smooth affine algebraic surface $\X = \{f(x) + yz = 0\} \subset \C^3$. What Manolescu constructed is an embedding
\begin{equation} \label{eq:manolescu}
\Y \longrightarrow \mathrm{Hilb}_m(\X),
\end{equation}
whose image is the open subset of those subschemes whose image under projection $x: \X \rightarrow \C$ is again of length $m$. We use slightly different coordinates on $\X$ than \cite[Section 2.4]{Manolescu}, but it is straightforward to adapt the explicit formulae given there. Inspection of those formulae shows that $\iota$ corresponds to exchanging $y$ and $z$. Take the fixed point set of the last-mentioned involution of $\X$, which is the affine hyperelliptic curve $\X^{\mathrm{inv}} = \{f(x) + y^2 = 0\} \subset \C^2$.

\begin{lemma} \label{th:mumford}
The restriction of \eqref{eq:manolescu} yields an embedding
\begin{equation} \label{eq:mumford}
\Y^{\mathrm{inv}} \longrightarrow \Sym_m(\X^{\mathrm{inv}}).
\end{equation}
Its image consists of those divisors $D = p_1 + \cdots + p_m$ which contain no fibre of the hyperelliptic involution, which means no pair $(x,y) + (x,-y)$.
\end{lemma}

\begin{proof}
This could be derived from the previous result, which is what happens in \cite[Section 7]{Manolescu}. However, restriction to the fixed point set actually simplifies the picture considerably, and we therefore prefer to give a self-contained account, starting with the description of the map. Take some point in $\Y^{\mathrm{inv}}$, and write the associated matrix from \eqref{eq:det2} as
\begin{equation}\label{eq:matrixdefn}
A(x) = \begin{pmatrix} W(x) & V(x) \\ V(x) & U(x) \end{pmatrix}.
\end{equation}
By construction $U$ and $W$ are monic of degree $m$, and their $x^{m-1}$ coefficients sum to zero; while $V$ has degree $\leq m-1$. Since $f(x) = UW - V^2$, the ideal
\begin{equation} \label{eq:uv-ideal}
(U(x), y-V(x)) \subset \C[x,y]
\end{equation}
contains $f(x)+y^2$, hence defines a subscheme of $\X^{\inv} \subset \C^2$. It is obvious that \eqref{eq:uv-ideal} has length $m$, and the same holds for its intersection with $\C[x] \subset \C[x,y]$. The latter property means that the image of our subscheme under the projection $x: \X^{\inv} \rightarrow \C$ is still of length $m$. Finally, since $\X^{\mathrm{inv}}$ is a smooth curve, the Hilbert scheme is the same as the symmetric product. This defines the map \eqref{eq:mumford}, and shows that its image lies in the previously described open subset.

In the reverse direction, take a degree $m$ effective divisor $D = p_1 + \cdots + p_m$ on $\X^{\mathrm{inv}}$ which contains no fibre of the hyperelliptic involution. Write $p_i = (x_i,y_i)$, and set $U(x) = \prod_i (x - x_i)$. If the $x_i$ are pairwise distinct, take the unique polynomial $V(x)$ of degree $\leq m-1$ for which $V(x_i) = y_i$. %explicitly
%\begin{equation}
%V(x) = \sum_i y_i \cdot \frac{\prod_{j\neq i} (x-x_j)}{\prod_{j\neq i} (x_i-x_j)}.
%\end{equation}
In general, several of the $x_i$ can coincide, but only if the corresponding $y_i$ are all nonzero and coincide as well. In that case, one can still find a unique $V$ which, at each such point, approximates the branch of $\sqrt{-f(x)}$ with value $y_i$ to order equal to the appropriate multiplicity. By construction, $f(x)+V^2$ is divisible by $U$, and we then define $W$ by $f(x)+V^2=UW$. Since $\mathrm{deg}(V) \leq m-1$, $W$ is necessarily monic of degree $m$. Finally, since the roots of $f$ have center of mass zero, the $x^{m-1}$ coefficients of the polynomials $U$ and $W$ must sum to zero. Then defining $A(x)$ as in \eqref{eq:matrixdefn} yields a point in $\Y^{\inv}$.
\end{proof}

\begin{remark} \label{th:relative-hilb}
The embedding described above has a natural extension to the whole of $\Slice^{\inv}$. Namely, let $\TT \rightarrow \Sym^0_{2m}(\C)$ be the family of double branched covers $f(x) + y^2 = 0$, where $f$ ranges over all monic polynomials of degree $2m$ with zero subleading coefficient. The total space $\TT$ is a smooth variety of dimension $2m$. To this, one can associate its relative Hilbert scheme $\HH = \mathrm{Hilb}_m(\TT / \Sym^0_{2m}(\C))$. More explicitly, $\TT$ is a subvariety of $\C^2 \times \Sym^0_{2m}(\C)$, and correspondingly $\HH$ lies inside $\Hilb_m(\C^2) \times \Sym^0_{2m}(\C)$. By the same formula as before, we then get an embedding fibered over $\Sym^0_{2m}(\C)$,
\begin{equation} \label{eq:hilb-embedding}
\Slice^{\inv} \longrightarrow \HH.
\end{equation}
\end{remark}

%To rephrase the last part in a less coordinate-dependent way, our divisor $D$ on $\X^{\inv}$ defines a %line bundle, whose pushforward via the (affine, of course) projection $\X^{\inv} \rightarrow \C$ is an %algebraically trivial rank two vector bundle, and in fact canonically trivialized. Multiplication by %$y$ on the fibres of that bundle yields an endomorphism of that trivial bundle, which is precisely %$A(x) \in \mathfrak{gl}_2(\C[x])$. The subscheme $(U(x), y-V(x))$ arises as a spectral curve for this %endomorphism. \qed

Let $\CC$ be the completion of $\X^{\inv}$ to a closed smooth algebraic curve, and $q_\pm$ its points at infinity. Inside the Jacobian $\Jac(\CC)$ of degree $m$ line bundles, take the theta divisor
\begin{equation} \label{eq:theta}
\Theta = \{ L \, : \, H^0(L(-q_+ - q_-)) \neq 0 \}.
\end{equation}
This makes sense since $\CC$ is of genus $m-1$, and $L(-q_+ - q_-)$ of degree $m-2$. Consider $\Y^{\mathrm{inv}}$ as a subset of $\Sym_m(\CC)$ via \eqref{eq:mumford}.

\begin{lemma} \label{th:aj}
The restriction of the Abel-Jacobi map $\Sym_m(\CC) \rightarrow \Jac(\CC)$ to $\Y^{\inv}$ is a fibre bundle over $\Jac(\CC) \setminus \Theta$ with fibre $\C^*$. Moreover, that fibre bundle is topologically trivial.
\end{lemma}

\begin{proof}
Take $D = p_1 + \cdots + p_m$, where $p_i \neq q_\pm$. If the line bundle $L = {\mathcal O}(D)$ lies in $\Theta$, there is a rational function $r$ which vanishes at $q_\pm$ and has poles of the appropriate orders at the points of $D$. Writing $p_i = (x_i,y_i)$, then clearly $r \prod_i (x-x_i)$ has poles of order $m-1$ at $q_\pm$, and none elsewhere, hence is of the form $g(x) + y h(x)$ for some polynomials $g$, $h$. Since $y$ has a pole of order $m$ at $q_\pm$, it follows that $h = 0$, so $r = g(x)/\prod_i (x-x_i)$, where $g$ must have degree $\leq m-1$. But then, $r$ either has a pole of order $\geq 2$ at a fixed point of the hyperelliptic involution, or two poles at points $(x,\pm y)$, hence $D$ cannot lie in the image of \eqref{eq:mumford}.

Conversely, suppose that $H^0(L(-q_+ - q_-)) = 0$, which by Riemann-Roch means $H^1(L(-q_+ - q_-)) = 0$. Then evaluation at $q_\pm$ defines an isomorphism $H^0(L) \rightarrow L_{q_+} \oplus L_{q_-}$. Hence, there is a $\C^*$ worth of divisors within the pencil $|L|$ which avoid $q_+$ and $q_-$. If such a divisor contains a fibre of the hyperelliptic involution, one can find a linearly equivalent one containing both points $q_\pm$, which is a contradiction to the original assumption. Hence, all our $\C^*$ family of divisors lie in the image of \eqref{eq:mumford}. This argument actually shows that $\Y^{inv} \rightarrow \Jac(\CC) \setminus \Theta$ can be identified with the complement of the two coordinate sections inside the $\C P^1$-bundle with fibres $\mathbb{P}(\C \oplus \mathit{Hom}(L_{q_+}, L_{q_-}))$. One can define such a $\C P^1$-bundle over the Jacobian for any two points of the curve $\CC$, and its topological type is independent of which points one chooses. In particular, by taking the two points to coincide, one sees that the bundle is topologically trivial (its restriction to $\Jac(\CC) \setminus \Theta$ must then be holomorphically trivial as well, since that space is affine and hence Stein).
\end{proof}

Lemmas \ref{th:mumford} and \ref{th:aj} are actually variations on arguments from \cite{Mumford}, the main difference being that \cite{Mumford} concerns hyperelliptic curves which have a branch point at infinity. %It is worth while to compare some topological features of $\Y^{\inv}$ and $\Sym_m(\X^{\inv})$:

\begin{lemma} \label{th:topology}
%Inclusion induces an isomorphism $H_1(\Y^{\inv}) \iso H_1(\Sym_m(\X^{\inv}))$.
%Moreover, $\pi_2(\Sym_m(\X^{\inv})) = 0$.
%Finally,
The first Chern class $c_1(\Sym_m(\X^{\mathrm{inv}}))$ is represented by the Poincar{\'e} dual of $-D$, where $D = \Sym_m(\X^{\inv}) \setminus \Y^{\inv}$.
\end{lemma}

\begin{proof} %It is a classical fact that $\Theta$ has singularities in codimension $2$, hence is irreducible, which means that its top-dimensional cohomology is $H^{2m-4}(\Theta) \iso \Z$. Using this and Poincar{\'e} duality, we get a long exact sequence
%\begin{equation}
%\cdots \rightarrow H_2(\Jac(\CC)) \rightarrow \Z \rightarrow H_1(\Jac(\CC) \setminus \Theta) %\rightarrow H_1(\Jac(\CC)) \rightarrow 0.
%\end{equation}
%The map to $\Z$ is the intersection number with $\Theta$, which is onto since that divisor is primitive in homology. From this and Lemma \ref{th:aj}, it follows that $H_1(\Y^{\inv}) \iso \Z^{2m-1}$. On the other hand, $\Sym_m(\X^{\mathrm{inv}})$ is homotopy equivalent to the $m$-skeleton of $(S^1)^{2m-1}$ by \cite{ong,kallel-salvatore}, in particular has fundamental group $\Z^{2m-1}$ and trivial second homotopy group. The map $H_1(\Y^{\inv}) \rightarrow H_1(\Sym_m(\X^{\inv}))$ is onto, since the difference between the two spaces has real codimension two. Hence, it must be an isomorphism (note that the corresponding statement for fundamental groups is not true: $\pi_1(\Sym_m(\X^{\inv})) \iso \Z^{2m-1}$ , but $\pi_1(\Y^{\inv})$ is nonabelian already for $m = 2$, by Lemma \ref{th:aj}).

%$c_1(\Y^{\inv}) = 0$ is clear from Lemma \ref{th:aj}, or alternatively it could be derived directly %from the definition as a complete intersection inside affine space.
$D$ is the image of a map $\C \times \Sym_{m-2}(\X^{\inv}) \rightarrow \Sym_m(\X^{\inv})$. Its closure is a map $\C P^1 \times \Sym_{m-2}(\CC) \rightarrow \Sym_m(\CC)$. Well-known facts about linear systems on hyperelliptic curves \cite[p.\ 13]{ACGH} imply that the image of this map is precisely the preimage of $\Theta$ under the Abel-Jacobi map. On the other hand, $c_1(\Sym_m(\CC))$ is represented by $2[\Sym_{m-1}(\CC) \times \{\mathit{point}\}]$ minus the preimage of $\Theta$ \cite[p.\ 337]{MacD}. Restricting to $\Sym_m(\X^{\inv})$ kills the first summand.
\end{proof}

\begin{remark} \label{th:pole}
As a variety cut out inside affine space (the space $\Slice^\inv$) by independent equations (the coefficients of $\chi$), $\Y^\inv$ carries an algebraic volume form, determined up to a constant. Lemma \ref{th:topology} says that its extension to $\Sym_m(\X^{\inv})$ has simple zeros along the irreducible divisor $D$, which means that it yields a trivialization of the canonical bundle of $\Sym_m(\X^{\inv})$ twisted by ${\EuScript O}(-D)$.
\end{remark}

\subsection{Symplectic forms and Lagrangian submanifolds}
In \cite{SS}, we equipped $\Slice$ with an exact K{\"a}hler form $\Omega$ as follows. Fix some $\alpha>m$, and for each $1 \leq k \leq m$ choose a strictly subharmonic function $\psi_k: \C \rightarrow \R$, such that $\psi_k(z) = |z|^{\alpha/k}$ at infinity. Apply $\psi_k$ to each entry of $A_k$, and let $\psi$ be the sum of the resulting terms, which is an exhausting plurisubharmonic function. Then set
\begin{equation} \label{eq:original-omega}
\Omega = -dd^c\psi.
\end{equation}
Any K{\"a}hler form defines a symplectic connection on the regular part of $\chi: \Slice \rightarrow \Sym_{2m}^0(\C)$. However, since the fibres are non-compact, one may not be able to integrate the associated horizontal vector fields to obtain parallel transport maps. This difficulty can be addressed by taking the horizontal vector fields and adding a multiple of the fibrewise Liouville vector field dual to $\Theta = -d^c\psi$. It is proved in \cite{SS} that in this way, one can get modified parallel transport maps defined on arbitrarily large compact subsets, which is enough for applications. In our case, since both $\Omega$ and $\Theta$ are $\iota$-invariant, the modified parallel transport is equivariant.

Let $\omega = d\theta$ be the restriction of $\Omega = d\Theta$ to our chosen fibre $\Y$, which lies over $t = (t_1,\dots,t_{2m}) \in \Conf_{2m}^0(\C)$. Parallel transport and vanishing cycle arguments can be used to construct certain Lagrangian submanifolds in $\Y$. Recall that a crossingless matching $\wp$ is the union of $m$ disjoint embedded closed arcs in $\C$, whose $2m$ endpoints are precisely the $t_i$. To each such $\wp$ one can associate a compact Lagrangian submanifold $L_\wp \subset \Y$.

\begin{lemma} \label{th:equivariant-vanishing}
For every $\wp$, $L_\wp \subset \Y$ is $\iota$-invariant. Moreover, there is a diffeomorphism $L_\wp \iso (S^2)^m$, under which $\iota|L_\wp$ corresponds to the involution $(p,q,r) \mapsto (q,p,r)$ of each factor $S^2 \subset \R^3$.
\end{lemma}

\begin{proof}
We need to recall the construction of the $L_\wp$ as relative vanishing cycles, which proceeds by induction on $m$ (see \cite{SS} for details). Let $\Slice'$ be the analogue of \eqref{eq:y-matrix} inside $\mathfrak{sl}_{2m-2}(\C)$, and $\chi': \Slice' \rightarrow \Sym_{2m-2}^0(\C)$ its adjoint quotient map. Fix $t' = (t_1',\dots,t_{2m-2}') \in \Conf_{2m-2}^0(\C)$, where in addition each $t_i'$ is assumed to be nonzero. Denote by $\Y'$ the fibre of $\chi'$ at that point, and by $\iota'$ the involution on it. Suppose that we have a crossingless matching $\wp'$ for $t'$, which consists of paths that avoid the origin, and let $L_{\wp'}$ be the associated Lagrangian submanifold. Next, consider the disc $D_\epsilon \subset \Sym_{2m}^0(\C)$ consisting of points $(t_1',\dots,t_{2m-2}',-\sqrt{z},\sqrt{z})$ for $|z| < \epsilon$. The restriction of $\chi$ to that disc is a holomorphic function
\begin{equation} \label{eq:chi-1}
\chi: \chi^{-1}(D_\epsilon) \longrightarrow D_\epsilon,
\end{equation}
which has a Morse-Bott nondegenerate critical submanifold in the fibre over $z = 0$. That submanifold can be canonically identified with $\Y'$. In fact, a neighbourhood of the critical submanifold inside $\chi^{-1}(D_\epsilon)$ can be identified with a neighbourhood of $\Y' \times \{0\} \subset \Y' \times \C^3$, in such a way that the map \eqref{eq:chi-1} is $(A,p,q,r) \mapsto p^2 + q^2 + r^2$. Finally, inspection of the construction shows that these identifications can be made $\iota$-equivariant, where the corresponding involution on $\Y' \times \C^3$ is
\begin{equation}
(A,p,q,r) \longmapsto (\iota'(A),q,p,r).
\end{equation}
Inside $\chi^{-1}(D_\epsilon)$, consider the subspace $W$ of all points that converge to a point of $L_{\wp'}$ under the negative gradient flow of the real part of \eqref{eq:chi-1}. That subspace is diffeomorphic to a neighbourhood of $L_{\wp'} \times \{0\} \subset L_{\wp'} \times \R^3$. Set $t = (t_1',\dots,t_{2m-2}',-\sqrt{\delta},\sqrt{\delta})$ for some small $\delta>0$, and let $\wp$ be the crossingless matching for $t$ obtained by adding an interval $[-\sqrt{\delta},\sqrt{\delta}]$ to $\wp'$. One produces $L_{\wp}$ by intersecting $W$ with the fibre of \eqref{eq:chi-1} over $\delta$.

If we make the induction assumption that $L_{\wp'}$ is $\iota'$-invariant, it follows from the construction that $L_{\wp}$ has the analogous property, and then the same holds for any crossingless matching with $2m$ strands by using parallel transport. Similarly, the proof in \cite[Section 5B]{SS} that $L_{\wp} \iso (S^2)^m$ carries through equivariantly, and implies the fact about $\iota|L_{\wp}$ stated above.
\end{proof}

Temporarily, let's specialize to the following situation. Take real numbers $\tau_1 < \cdots < \tau_m$ which add up to zero, and consider the polydisc $D^m_\epsilon \subset \Sym_{2m}^0(\C)$ consisting of points
\begin{equation} \label{eq:in-pairs}
(\tau_1 + \sqrt{z_1}, \tau_1 - \sqrt{z_1}, \dots, \tau_m + \sqrt{z_m}, \tau_m - \sqrt{z_m})
\end{equation}
for $|z_i| < \epsilon$. Let $t$ be the point obtained by taking all $z_i = \delta > 0$, where $\delta$ is small, and $\wp$ the crossingless matching which consists of the intervals $[\tau_i - \sqrt{\delta}, \tau_i + \sqrt{\delta}]$.

\begin{lemma} \label{th:framed-vanishing}
There is a contractible totally real half-dimensional submanifold $\Delta \subset \chi^{-1}(D^m_\epsilon)$ with the following property. The restriction of $\chi$ to $\Delta$ yields a map $\Delta \rightarrow [0,\epsilon)^m$, which has $(\delta,\dots,\delta)$ as a regular value. Moreover, the intersection of $\Delta$ with the fibre at that point is precisely $L_\wp$.
\end{lemma}

\begin{proof}
There is a special point in the fibre over $(0,\dots,0) \in D^m_\epsilon$, and local coordinates $(u_1,v_1,w_1,\dots,u_m,v_m,w_m)$ on $\chi^{-1}(D_\epsilon^m)$ centered at that point, in which $\chi = (u_1^2+v_1^2+w_1^2,\dots,u_m^2+v_m^2+w_m^2)$. Take the real locus in those coordinates. Then, its intersection with the fibre over $(\delta,\dots,\delta)$ is isotopic to $L_\wp$ as a totally real submanifold inside that fibre, see \cite[Lemma 30]{SS}. One can easily extend that isotopy to yield $\Delta$ with the desired properties.
\end{proof}

Perutz (\cite{Perutz}, building on results of Varouchas \cite{V1}) provided a direct construction of K{\"a}hler forms on symmetric products, which we will now adapt to our purpose. Choose a K{\"a}hler form $\alpha$ on $\X^{\inv}$, which should be invariant under the hyperelliptic involution. Let $N \subset \Conf_m(\X^{\inv})$ be a relatively compact open subset, which we think of as a subset of $\Sym_m(\X^{\inv})$ whose closure is disjoint from the big diagonals. Then \cite[Proposition 7.1]{Perutz} there is an exact K{\"a}hler form on $\Sym_m(\X^{\inv})$, whose restriction to $N$ agrees with the product of $m$ copies of $\alpha$. We may additionally assume that this K{\"a}hler form is still invariant under the hyperelliptic involution, and choose a one-form primitive with the same property. Pull both back to $\Y^{\inv}$ via \eqref{eq:mumford}, and denote the result by $\omega_P = d\theta_P$.

For each embedded path in $\C$ whose endpoints belong to $t = (t_1,\dots,t_{2m})$, and which otherwise avoids $t$, there is an associated simple closed curve in $\X^{\inv}$, invariant under the hyperelliptic involution. In particular, a crossingless matching yields $m$ disjoint such curves, which then give rise to a torus $T_{\wp} \subset \Conf_m(\X) \cap \Y^{\inv} \subset \Y^{\inv}$ in the way familiar from Heegaard Floer theory. If the subset $N$ in the previous construction is chosen so that it contains $T_{\wp}$, that torus will be Lagrangian with respect to $\omega_P$. Since the hyperelliptic involution acts by $-1$ on $H^1(T_{\wp};\R)$ but preserves $\theta_P$, the torus is also automatically exact.

We also find it useful to mention a variant of this construction, which combines Perutz' ideas with ones of Manolescu \cite[Section 4]{Manolescu}. Choose a K{\"a}hler form $\beta$ on $\C^2$, which should be invariant under the involution
\begin{equation} \label{eq:xy}
(x,y) \longmapsto (x,-y).
\end{equation}
Let $O \subset \Conf_m(\C^2)$ be a relatively compact open subset, which we think of as a subset of $\Hilb_m(\C^2)$ whose closure is disjoint from the preimage of the big diagonal under the Hilbert-Chow map $\Hilb_m(\C^2) \rightarrow \Sym_m(\C^2)$. Again using results of Varouchas \cite{V1,V2}, one can construct a K{\"a}hler form on $\Hilb_m(\C^2)$ whose restriction to $O$ agrees with the product of $m$ copies of $\beta$. Take the product of that form with an arbitrary K{\"a}hler form on $\Sym_{2m}^0(\C) \iso \C^{2m-1}$, and pull that back via the embedding
\begin{equation} \label{eq:2emd}
\Slice^{\inv} \longrightarrow \HH \longrightarrow \Hilb_m(\C^2) \times \Sym_{2m}^0(\C)
\end{equation}
from Remark \ref{th:relative-hilb}. We denote the result by $\Omega_M$, and write it as $\Omega_M = d\Theta_M$ for some one-form $\Theta_M$. As before, we may assume that both $\Omega_M$ and $\Theta_M$ are invariant under the involution which corresponds to \eqref{eq:xy} under \eqref{eq:2emd}, concretely
\begin{equation} \label{eq:hyperelliptic-involution}
(U,V,W) \longmapsto (U,-V,W).
\end{equation}
The restrictions of $\Omega_M$ and $\Theta_M$ to $\Y^{\inv}$ are written as $\omega_M$ and $\theta_M$, respectively. By the same argument as for $\omega_P$, the tori $T_{\wp}$ are exact Lagrangian with respect to $\omega_M$, provided that the subset $O$ has been chosen sufficiently large.

Both $\omega_P$ and $\omega_M$ share the common feature of being in product form on a large subset of $\Y^{\inv} \subset \Sym_m(\X^{\inv})$ which is disjoint from the big diagonal. Also, both forms extend to $\Sym(\X^{\inv})$, but the extension of $\omega_P$ is still exact, while that of $\omega_M$ is not (for the last-mentioned fact, see work in progress by Lekili-Perutz). On the other hand, $\omega_M$ comes from a K{\"a}hler form defined on the whole of $\Slice^{\inv}$, while $\omega_P$ does not.

\begin{lemma} \label{th:interpolate}
Fix a crossingless matching $\wp$. There is an isotopy of embedded tori $L_r \subset \Y^{\inv}$ interpolating between $L_0 = L_{\wp}^{\inv}$ and $L_1 = T_{\wp}$. Moreover, each $L_r$ is exact Lagrangian with respect to an appropriate exact symplectic form $\omega_r = d\theta_r$.
\end{lemma}

\begin{proof}
Suppose first that our chosen point $t$ is of the form \eqref{eq:in-pairs}, where $\tau \in \Conf_m^0(\C)$ and the $z_k \in \C^*$ are small, and that $\wp$ consists of the union of the straight paths joining $\tau_k \pm \sqrt{z_k}$. Recall that $L_{\wp}^{\inv}$ is constructed as an iterated vanishing cycle for $\Omega|\Slice^{\inv}$. In this particular case (compare \cite[Lemma 30]{SS} or Lemma \ref{th:framed-vanishing} above), the vanishing cycle construction can be carried out entirely within a neighbourhood of a particular point in $\Y^{\inv}$, whose image under $\Slice^{\inv} \rightarrow \HH \rightarrow \Hilb_m(\C^2) \rightarrow \Sym_m(\C^2)$ is $(x_1,y_1) = (\tau_1,0), \dots, (x_m,y_m) = (\tau_m,0)$. Importantly, that point does not lie on the big diagonal. It is then easy to see that if we use $\Omega_M$ instead, we get an iterated vanishing cycle which lies in $\Conf_m(\X^{\inv}) \cap \Y^{inv}$, and which is the product of $m$ disjoint simple closed curves invariant under the hyperelliptic involution, isotopic to the preimages of the arcs in $\wp$. Such a torus is then necessarily exact Lagrangian isotopic to $T_{\wp}$ (see \cite[Proposition 4.3]{Manolescu}, where a parallel argument is carried out in the context of the whole space $\Y$ rather than the fixed point set $\Y^{\inv}$). Interpolating linearly between $\Omega|\Slice^{\inv}$ and $\Omega_M$ then yields a family of vanishing cycles which provides the rest of the desired isotopy.

In order to derive the general case from this, we'd like to use parallel transport, which means that we again have to modify our K{\"a}hler forms. Take a function $h: \R \rightarrow \R$ such that $h''(\rho) \geq 0$ everywhere, and
\begin{equation}
h'(\rho) \begin{cases} = 0 & \rho \leq R_0, \\ > 0 & R_0 < \rho < R_1 \\ = 1 & \rho \geq R_1.
\end{cases}
\end{equation}
where $R_0 < R_1$ are large positive constants. Replacing $\psi$ by $h(\psi)$ in \eqref{eq:original-omega} yields a two-form which is nonnegative on complex tangent lines, which vanishes on a large compact subset, and which outside an even larger compact subset agrees with $\Omega$. Following \cite[Equation (23)]{Manolescu} we take
\begin{equation}
\tilde\Omega_M = \epsilon \Omega_M - dd^c h(\psi)|\Slice^{\inv}
\end{equation}
for some small $\epsilon>0$. To understand the properties of this form, consider first its restriction $\tilde{\omega}_M$ to $\Y^{\inv}$. On a large compact subset $\tilde\omega_M = \epsilon \omega_M$, hence the tori $T_{\wp}$ are still Lagrangian (and exact Lagrangian, if one uses the obvious one-form primitive). On the other hand, at infinity $\tilde\omega_M = \epsilon \omega_M + \omega|\Y^{\inv}$ is close to $\omega|\Y^{\inv}$. Fix a sufficiently large level set of the original plurisubharmonic function $\psi|\Y^{\inv}$. This is regular, by \cite[Lemma 41]{SS}, and of contact type with respect to $\omega|\Y^{\inv}$. Hence, it is also of contact type for $\tilde{\omega}_M$ if we choose $\epsilon$ sufficiently small. One can arrange that the same holds not just in a single fibre, but over a finite path in the base of the fibration $\Slice^{\inv} \cap \chi^{-1}(\Conf_{2m}^0(\C)) \rightarrow \Conf_{2m}^0(\C)$. This is enough to define symplectic parallel transport along this path, at least on a large compact subset of the fibres which contains our Lagrangian tori (compare \cite[Section 6a]{khovanov-seidel}). Moreover, the same will be true for the family of symplectic forms which interpolates linearly between $\Omega|\Slice^{\inv}$ and $\tilde{\Omega}_M$.

With this in mind, the strategy for proving the general case is as follows. We find a path in $\Conf_{2m}^0(\C)$ which connects our given $t = t_0$ to some other point $t_1$ of the form \eqref{eq:in-pairs}, and a smooth family $\wp_s$ of crossingless matchings with endpoints $t_s$, which deforms $\wp = \wp_0$ to the collection of straight paths $\wp_1$ joining the points of $t_1$ in the way described before. Each fibre $\Y_s^{\inv} = \chi^{-1}(t_s) \cap \Slice^{\inv}$ can be equipped with the one-parameter family of K{\"a}hler forms $\omega_{r,s}$ obtained by restricting $(1-r)\Omega|\Slice^{\inv} + r \tilde\Omega_M$. For $s = 1$ and arbitrary $r$, we have a family of tori in $\Y_s^{\inv}$ which are Lagrangian with respect to $\omega_{r,s}$ (this is the isotopy provided by the previously proved special case). On the other hand, the Heegaard Floer construction provides another family of tori, for arbitrary $s$ and $r = 1$. We pull all these tori back to $\Y^{\inv} = \Y_0^{\inv}$ by symplectic parallel transport. This yields an isotopy connecting $T_{\wp}$ with $L_{\wp}^{\inv}$ (since the latter is by definition compatible with parallel transport for $\Omega|\Slice^{\inv}$).
\end{proof}

\subsection{Floer cohomology}

The spaces $\Y$ are smooth affine varieties, hence satisfy the convexity condition from Section \ref{subsec:borel-floer} (take a sequence of balls in $\Slice \iso \C^{4m-1}$, and intersect it with $\Y$ to form the required exhaustion). Moreover, the symplectic form $\omega$ is exact, and the Lagrangian submanifolds $L_\wp$ are compact and (since they are simply-connected) necessarily exact. Clearly, the same property carries over to the fixed parts $L_{\wp}^{\inv}$ inside $\Y^{\inv}$. Hence, for any two crossingless matchings $\wp_\pm$, we have a well-defined Floer cohomology $HF(L_{\wp_+},L_{\wp,-})$, as well as the equivariant version $HF_{\mathrm{borel}}(L_{\wp_+},L_{\wp_-})$, and finally Floer cohomology taken inside the fixed locus, $HF(L_{\wp_+}^{\inv},L_{\wp_-}^{\inv})$.

\begin{lemma} \label{th:stablynormalexists}
The pair $(L_{\wp_+},L_{\wp_-})$ in $\Y$ has stably trivial normal structure (in fact, there is a preferred homotopy class of stable normal trivializations).
\end{lemma}

\begin{proof}
Since $\Y$ is the regular fibre of a map $\C^{4m-1} \rightarrow \C^{2m-1}$, its tangent bundle carries a stable trivialization, which is unique up to homotopy (since it depends on splitting a short exact sequence of complex differentiable vector bundles). The same argument shows that parallel transport maps respect the stable trivialization, up to canonical homotopy.

Consider a crossingless matching $\wp$ as in Lemma \ref{th:framed-vanishing}, and let $\Delta$ be the totally real submanifold constructed there. We have a commutative diagram of vector bundles over $L_{\wp}$,
\begin{equation} \label{eq:triv-triv}
\xymatrix{
0 \ar[r] & \ar[d]^{\iso} TL_{\wp} \otimes_\R \C \ar[r] & \ar[d]^{\iso} (T\Delta|L_{\wp}) \otimes_\R \C \ar[r] & L_{\wp} \times \C^m \ar[r] \ar[d]^{=} & 0 \\
0 \ar[r] & T\Y|L_{\wp} \ar[r] & T(\chi^{-1}(D^m_\epsilon))|L_{\wp} \ar[r] & L_{\wp} \times \C^m \ar[r] & 0
}
\end{equation}
Since $\Delta$ is contractible, $T\Delta$ is trivial, and this induces a stable trivialization of $TL_{\wp}$. Diagram-chasing in \eqref{eq:triv-triv} shows that the complexification of this trivialization is canonically homotopic to the trivialization of $T\Y|L_{\wp}$ constructed before. Because of our previous remarks about parallel transport, the same is true for any crossingless matching.

The same argument goes through equivariantly. More explicitly, $T\Y$ has an equivariant stable trivialization, $TL_{\wp}$ has an equivariant stable trivialization, and under the natural isomorphism $TL_{\wp} \otimes \C \iso T\Y|L_{\wp}$, there is an equivariant homotopy between these two trivializations, which in turn is determined uniquely up to homotopies. Restriction to fixed point sets and to the anti-invariant directions there now yields the desired normal trivializations.
\end{proof}

Our general theory from Section \ref{subsec:Floerlocalize} now provides a localization map
\begin{equation} \label{eq:lo}
HF_{\mathrm{borel}}(L_{\wp_+}, L_{\wp_-}) \longrightarrow HF(L_{\wp_+}^{\inv}, L_{\wp_-}^{\inv})[[q]],
\end{equation}
which becomes an isomorphism after inverting $q$, and an inequality of ranks
\begin{equation} \label{eq:smith-applied}
\dim\, HF(L_{\wp_+}, L_{\wp_-}) \geq \dim\, HF(L_{\wp_+}^{\inv}, L_{\wp_-}^{\inv}).
\end{equation}

\begin{lemma} \label{th:hfhf}
$HF(L_{\wp_+}^{\inv},L_{\wp_-}^{\inv}) \iso HF(T_{\wp_+},T_{\wp_-})$.
\end{lemma}

This follows from Lemma \ref{th:interpolate} applied to both crossingless matchings (one checks easily that this can be done, meaning that the symplectic forms can be taken to be the same in both isotopies), and standard invariance properties of Floer cohomology. Note that by construction, $\omega_P$ and $\theta_P$ extend to $\Sym_m(\X^{\inv})$. Hence, we can also form the Floer complex of the $T_{\wp_\pm}$ inside that larger space. We will denote the differential which defines this theory by $\bar{d}$, and the resulting cohomology by $\overline{HF}(T_{\wp_+},T_{\wp_-})$.

Up to now we have worked with ungraded groups, but the next step will require at least relative gradings. Both $\Y$ and $\Y^{\inv}$ admit algebraic volume forms (since they are complete intersections inside affine spaces). Since the $L_{\wp}$ are simply-connected, they admit gradings, so $HF(L_{\wp_+},L_{\wp_-})$ is $\Z$-graded in a way which is unique up to a constant. Through the stable normal trivialization, the fixed parts $L_{\wp}^{\inv}$ inherit gradings, so $HF(L_{\wp_+}^\inv,L_{\wp_-}^\inv)$ is again $\Z$-graded up to a constant. This carries over to $HF(T_{\wp_+},T_{\wp_-})$ via Lemma \ref{th:hfhf}, but it is no longer true for $\overline{HF}(T_{\wp_+},T_{\wp_-})$. Namely, consider a Floer trajectory $u$ which contributes to $\bar{d}$, having endpoints $x_\pm$, and which intersects the divisor $D$ with total multiplicity $k$. One can arrange that the almost complex structures used are standard near $D$, so that necessarily $k \geq 0$, with equality iff $u$ avoids $D$ altogether. Because of the behaviour of the algebraic volume form, see Remark \ref{th:pole}, the difference between the indices of $x_\pm$ in the given grading is $2k + 1$. In other words, one can write $\bar{d} = d_0 + d_1 + \cdots$, where $d_k$ has degree $2k + 1$, and where the $k = 0$ term is the differential which yields $HF(T_{\wp_+},T_{\wp_-})$. A standard filtration and spectral sequence argument then shows that:
\begin{equation} \label{eq:smith-added}
\dim\, HF(T_{\wp_+}, T_{\wp_-}) \geq \dim\, \overline{HF}(T_{\wp_+}, T_{\wp_-}).
\end{equation}

\subsection{Link invariants}

Fix the crossingless matching $\wp$ corresponding to a sequence of $m$ nested horseshoes lying in the upper half-plane, with endpoints $t = (t_1,\dots,t_{2m})$ on the real line. As discussed before, we can associate to this a Lagrangian submanifold $L_{\wp} \subset \Y$. A braid $\beta$ on $m$ strands gives rise to a loop $\beta \times \id \in \mathit{Br}_{2m}=\pi_1(\Conf_{2m}^0(\C))$ based at $t$. Hence, via (rescaled) parallel transport maps, we get another Lagrangian submanifold, denoted by $(\beta \times \id)(L_{\wp})$ for simplicity. On the other hand, we can form the closure of $\beta$, by taking the braid described by $\beta \times \id$ in $\R^2 \times [0,1]$, and capping strands with a copy of $\wp$ above and its reflection below. This yields an oriented link $\kappa = \kappa_\beta$. Recall the following from \cite{SS}:

\begin{definition} \label{def:kh-symp}
The symplectic Khovanov cohomology $\Kh_{\mathrm{symp}}^*(\kappa)$ is the Lagrangian Floer cohomology $HF^{*+m+w}(L_{\wp}, (\beta \times \id)(L_{\wp}))$.
\end{definition}

Here, we turn $L_{\wp}$ into a graded Lagrangian submanifold, and then use the canonical lifts of monodromy maps to graded symplectic diffeomorphisms to get an induced grading of $(\beta \times \id)(L_{\wp})$. These determine an absolute grading of Floer cohomology, which is independent of all choices. We shift this grading by the number of strands and by the writhe $w$ of $\beta$. The main theorem of \cite{SS} is that this group is invariant under the Markov moves, and hence indeed depends only on the oriented link itself. The proof of this theorem goes over routinely to the equivariant case, and shows that the equivariant symplectic Khovanov cohomology, defined as
\begin{equation}
\Kh_{\mathrm{symp,eq}}^*(\kappa) \stackrel{\mathrm{def}}{=} HF^{*+m+w}_{\mathrm{borel}}(L_{\wp}, (\beta \times \id)(L_{\wp})),
\end{equation}
is an oriented link invariant. Note that unlike the case of $\Kh_{\mathrm{symp}}$ itself, we do not at present have a conjectural analogue for this construction in the framework of combinatorial Khovanov cohomology. Similarly, we can work with the fixed parts of our Lagrangian submanifolds, and define a fixed part symplectic Khovanov cohomology
\begin{equation} \label{eq:kh-fixed}
\Kh_{\mathrm{symp,inv}}^*(\kappa) \stackrel{\mathrm{def}}{=} HF^{*+(m+w)/2}(L_{\wp}^{\inv}, (\beta \times \id)(L_{\wp}^{\inv})).
\end{equation}
Again, direct imitation of the arguments of \cite{SS} implies this is an oriented link invariant. Note that the grading is now either in $\Z$ or $\Z + \frac{1}{2}$, depending on the parity of $m+w$ (which is the number of components of $\kappa$ modulo $2$). To explain the occurrence of $(m+w)/2$, recall that the grading shift in the original construction of $\Kh_{\mathrm{symp}}$ arose when discussing invariance under Markov $II^-$ moves, see \cite[Lemma 57]{SS}, and had to do with the Morse index of a local maximum of a Morse function on $S^2$. In the fixed point theory, the analogous discussion applies to a Morse function on the circle $S^1$. What one finds is that if braids $\beta \in Br_m$ and $\tilde\beta_\pm = \sigma_m^{\pm 1}\beta \in Br_{m+1}$ are related by adding a positive respectively negative half-twist in the additional strand, and if $\wp$ respectively $\tilde{\wp}$ denote the standard crossingless matchings, then
\begin{equation}
\begin{aligned}
& HF^*(L_{\wp}^{\inv}, (\beta \times \id)(L_{\wp}^{\inv})) \iso HF^*(L_{\tilde{\wp}}^{\inv}, (\tilde\beta_+ \times \id)(L_{\tilde{\wp}}^{\inv})), \\
& HF^*(L_{\wp}^{\inv}, (\beta \times \id)(L_{\wp}^{\inv})) \cong HF^{*+1}(L_{\tilde{\wp}}^{\inv}, (\tilde\beta_- \times \id)(L_{\tilde{\wp}}^{\inv})).
\end{aligned}
\end{equation}
The discrepancy for Markov $II^-$ relative to Markov $II^+$ is compensated for by shifting the grading by $1/2$ each way in \eqref{eq:kh-fixed}. Note that in addition to the various Floer cohomology theories, we have the localization map \eqref{eq:lo}. It is easy to see that this is invariant under Markov $I$ moves. The expected behaviour is that it remains unchanged under Markov $II^+$ moves, and gets multiplied by $q$ under Markov $II^-$. We have not checked the details, since there are no applications at present (however, see Section \ref{subsec:speculate} below for some speculative ideas in that direction).

Let $N_\kappa$ be the double cover of $S^3$ branched along $\kappa$. If we consider the tori $T_{\wp}$, $(\beta \times \id)(T_{\wp})$ as lying inside $\Sym_m(\X^{\inv})$, they are the Heegaard tori for a Morse function on $N_\kappa \# S^1 \times S^2$ with two local minima and maxima, as in \cite{LL}. Moreover, these tori are weakly admissible, compare \cite[Proposition 7.4]{Manolescu}. Hence, the Floer cohomology groups previously denoted by $\overline{HF}(T_{\wp}, (\beta \times \id)(T_{\wp}))$ are just the Heegard Floer groups $\widehat{HF}(N_\kappa \# S^1 \times S^2)$ from \cite{OS-HolDisc}, with coefficients in $\Z_2$. We also have (see for instance \cite[Theorem 2.4]{LL})
\begin{equation}
\widehat{HF}(N_\kappa \# S^1 \times S^2) \iso \widehat{HF}(N_\kappa) \otimes H^*(S^1;\Z_2).
\end{equation}
With this in mind, the Smith inequality \eqref{eq:smith-applied} and the filtration argument \eqref{eq:smith-added} yield the following:

\begin{Corollary}
For any oriented link $\kappa \subset S^3$, we have the following inequality of ranks between cohomology groups with $\Z_2$ coefficients:
\begin{equation} \label{eq:big-inequality}
\dim \Kh_{\mathrm{symp}}(\kappa) \geq 2\dim \widehat{HF}(N_\kappa). \qed
\end{equation}
\end{Corollary}

$\widehat{HF}(N_\kappa)$ breaks up into summands indexed by $Spin^c$ structures. The Euler characteristic of each summand is $\pm 1$ if $H_1(N_\kappa)$ is finite, or zero if it is infinite \cite[Proposition 5.1]{OS-HolDisc2}. Moreover, the $Spin^c$ structures themselves form an affine space over $H_1(N_\kappa)$. This and \eqref{eq:big-inequality} immediately imply Corollary \ref{cor:mainone} from the Introduction.

\subsection{Speculations} \label{subsec:speculate}

It is possible that the Floer cohomology groups inside $\Y^{\inv}$ and $\Sym_m(\X^{\inv})$ always coincide, but we cannot prove that (a claimed proof of this, which was present in an earlier version of this preprint, contained an error concerning the general symplectic geometry properties of embeddings of affine varieties). A particularly interesting aspect of this relationship is the grading. Suppose that $\kappa$ is a knot. Then each summand of $\widehat{HF}(N_\kappa)$ corresponding to a given $\mathit{Spin}^c$ structure carries a relative $\Z$-grading \cite{OS-HolDisc}. In \cite{OS-fourmfld} those gradings were lifted to an absolute $\Q$-grading. On the other hand, our groups $\Kh_{\mathrm{symp,inv}}^*(\kappa)$ come with an absolute $\Z + 1/2$-grading. Both gradings are of Maslov type, meaning that the difference between two generators joined by a Floer-type disc equals the Maslov index of that disc. Hence, if one decomposes $\Kh_{\mathrm{symp,inv}}^*(\kappa)$ into direct summands indexed by connected components of the relevant path space in $\Y^{\inv}$, the difference between the two absolute gradings is constant on each component (we ignore the components which contribute trivially to Floer cohomology). The set of rational numbers encountered as differences could potentially be a knot invariant (and if one restricts to the components corresponding to $\mathit{Spin}$ rather than $\mathit{Spin^c}$ structures, even a knot concordance invariant, following a similar philosophy to that of  \cite{Conc}).

\begin{example}
Let $\kappa$ be the left-handed trefoil.  A variant of the argument at the end of \cite{SS} shows that
\begin{equation}
\Kh_{\mathrm{symp,inv}}^*(\kappa) \cong H^{*-\frac{1}{2}}(S^1) \oplus
H^{*-\frac{1}{2}}(S^1) \oplus H^{*-\frac{1}{2}}(S^1).
\end{equation}
The Ozsv\'ath-Szab\'o $\Q$-degrees of the corresponding generators of $\widehat{HF}^*(N_\kappa)$ are $-\frac{1}{2}$, $\frac{1}{6}$, $\frac{1}{6}$, and the connect sum with $S^1 \times S^2$ produces associated pairs of generators whose degrees are further shifted by $\pm \frac{1}{2}$. Hence, the grading differences are $\frac{3}{2}$, $\frac{5}{6}$, $\frac{5}{6}$. The first number is the one corresponding to the unique $\mathit{Spin}$ structure.
 \end{example}

Another possible object of study would be refined invariants obtained from the localization map \eqref{eq:lo}, in particular the cokernel polynomials \eqref{equation:cokernelpoly}.

\begin{conjecture} \label{conj:jones}
Suppose that $\kappa = \kappa_{\beta}$ is the link closure of a braid $\beta \in Br_m$ with writhe $w$. The polynomial $(-1)^{m+w}t^{(m+w)/2} P_{\cok}(t)$ is an oriented link invariant. For quasi-alternating links, this coincides with the normalised Jones polynomial:
\begin{equation}
V_{\kappa}(s) \ = \ \frac{(-1)^{m+w}t^{(m+w)/2} P_{\cok}(t)}{t^{1/2} + t^{-1/2}}\, \Big|_{s=t^{-1}}
\end{equation}
\end{conjecture}

Recall that $P_{\cok}(t)$ measures the complexity of the torsion $\Z_2[[q]]$-module which is the cokernel of the renormalized localization map. The additional expression involving $m+w$ compensates for the change in this map under Markov $II^-$. As mentioned in the Introduction, with $\Z_2$ coefficients  there is a (currently unpublished, but related to work of Rezazadegan \cite{Reza}) spectral sequence going from combinatorial Khovanov cohomology to its symplectic counterpart. On the other hand, the total rank of combinatorial Khovanov cohomology agrees with that of $\widehat{HF}(N_\kappa \# S^1 \times S^2)$ in the quasi-alternating case. In view of our results, this implies that, over $\Z_2$, $\Kh_{\mathrm{symp}}(\kappa)$ and $\Kh_{\mathrm{symp,inv}}(\kappa)$ both have rank which equals that of combinatorial Khovanov cohomology. In fact, both theories are then supposed to be versions of combinatorial Khovanov cohomology, but with the bigrading collapsed in different ways, analogously to \cite[Corollary 1.5]{khovanov-seidel}. Conjecture \ref{conj:jones} is an attempt to recover aspects of that bigrading by a comparison argument. Unfortunately, it is very unlikely that such an attempt would work for general knots and links.

\section{Symmetric links}
Take $\Sym_{2m}(\C)$ and equip it with the involution which takes $t = (t_1,\dots,t_{2m})$ to $-t = (-t_1,\dots,-t_{2m})$. The fixed point set is the subset $\Sym_{2m}(\C)^{\inv}$ of those $t$ which can be written in such a way that $t_{2i-1} = -t_{2i}$. There is a natural isomorphism
\begin{equation} \label{eq:sym-quotient}
\Sym_{2m}(\C)^{\inv} \iso \Sym_m(\C)
\end{equation}
which takes a point $t$ in the form written above to $\bar{t} = (t_2^2, t_4^2,\dots, t_{2m}^2)$. Similarly, one defines the subspace $\Conf_{2m}(\C)^{\inv}$ of symmetric configurations, which is identified with $\Conf_m(\C^*)$ by restricting \eqref{eq:sym-quotient}.

Without changing notation, we will consider the slice $\Slice$ consisting of matrices of the form \eqref{eq:y-matrix} but inside $\mathfrak{gl}_{2m}(\C)$, which means that $\mathrm{trace}(A_1)$ no longer needs to be zero. So $\Slice \iso \C^{4m}$, and the adjoint quotient map $\chi$ goes to $\Sym_{2m}(\C) \iso \C^{2m}$. Otherwise, its properties are essentially unchanged. In particular, if one takes $t \in \Conf_{2m}^0(\C)$, then the fibre $\Y = \chi^{-1}(t)$ is the same as before, and the other smooth fibres are diffeomorphic to it through parallel transport. Assume that $m$ is even, and consider the following involution on $\Slice$:
\begin{equation} \label{eq:minus-one}
\iota(A_1,\dots,A_m) = (-A_1, A_2, -A_3, A_4, \dots, -A_{m-1},A_m).
\end{equation}
This covers the involution $t \mapsto -t$ of $\Sym_{2m}(\C)$. Moreover, if $t \in \Conf_{2m}(\C)^{\inv}$, then the fixed part $\Y^{\inv}$ of the fibre $\Y = \chi^{-1}(t)$ can be identified with $\bar{\Y} = \bar\chi^{-1}(\bar{t})$, where $\bar{t}$ corresponds to $t$ under \eqref{eq:sym-quotient}, and $\bar\chi: \bar\Slice \rightarrow \Sym_m(\C)$ is the slice and adjoint quotient map for $\mathfrak{gl}_m(\C)$. We write this fact concisely as
\begin{equation} \label{eq:y-inv}
\Y^{\inv} = \bar{\Y}.
\end{equation}

One can arrange for the symplectic form $\Omega$ on $\Slice$ to be invariant under \eqref{eq:minus-one}. Its restriction to the fixed point set yields a form $\bar{\Omega}$ on $\bar\Slice$ in the same general class. Hence, the rescaled parallel transport map associated to paths in $\Conf_{2m}(\C)^{\inv}$ will be equivariant, and their restrictions to fixed point sets reduce to the corresponding maps for the associated paths in $\Conf_m(\C^*) \subset \Conf_m(\C)$. Next, fix some $t \in \Conf_{2m}(\C)^{\inv}$, and a crossingless matching $\wp$ which is symmetric in the same sense. Its image under $x \mapsto x^2$ is a crossingless matching $\bar{\wp}$ with endpoints in $\bar{t}$, which additionally avoids the origin. The associated Lagrangian submanifolds are related as follows:

\begin{lemma}
Under the identification \eqref{eq:y-inv}, the fixed part $L_{\wp}^{\inv}$ is Lagrangian isotopic to $L_{\bar{\wp}}$.
\end{lemma}

This follows from the construction as vanishing cycles, approached directly as in \cite[Lemma 30]{SS} (the inductive construction from Lemma \ref{th:equivariant-vanishing} would not work here because of the symmetry constraints: if we move the points of $\bar{t}$ so that two of them come together, four of the points of the corresponding $t$ will come together in pairs, and the points where they come together cannot lie at the origin).

\begin{lemma}
For any two symmetric crossingless matchings $\wp_\pm$, the pair $(L_{\wp_+},L_{\wp_-})$ in $\Y$ has stably trivial normal structure.
\end{lemma}

The proof is by the same argument as in Lemma \ref{th:stablynormalexists}. As a consequence, we now have an inequality
\begin{equation} \label{eq:wp-bar}
\dim\, HF(L_{\wp_+},L_{\wp_-}) \geq \dim\, HF(L_{\bar{\wp}_+}, L_{\bar{\wp}_-}).
\end{equation}

To apply this to knot theory, we set $t$ to be the collection of $2m$-th roots of unity, and $\wp$ the crossingless matching comprising the arcs of the unit circle $\{\exp(i\pi [k/m, (k+1)/m]) \ | \ k=0,2,4,\ldots, 2m-2\}$. For any braid $\beta \in \pi_1(\Conf_{2m}(\C))$ with base point $t$, we again have two Lagrangian submanifolds $L_{\wp}$ and $\beta(L_{\wp})$.

\begin{lemma} \label{th:wal}
As an ungraded group, $HF(L_{\wp},\beta(L_{\wp}))$ is the symplectic Khovanov cohomology of the link formed by taking the circular plat closure of $\beta$. \qed
\end{lemma}

Here, the circular plat closure is the link obtained from the graph of $\beta$ in $\R^2 \times [0,1]$ by attaching a copy of $\wp$ to the top and bottom. Lemma \ref{th:wal} is a special case of a more general result of Waldron \cite[Section 4.2]{waldron}, who gives a general way of computing symplectic Khovanov cohomology for links in admissible position in $\R^3$ (ones for which the height function given by the $z$ coordinate is a Morse function with all local maxima in the half-space $\{z>0\}$ and all local minima in the half-space $\{z<0\}$).

From now on, suppose that $\beta$ is symmetric, which means that is is given by a path in $\Conf_{2m}(\C)^\inv$. Then, the circular plat closure is a link $\kappa \subset \R^3$ disjoint from the $z$-axis. Taking the quotient under $(x,y,z) \mapsto (-x,-y,z)$ then yields another link $\bar\kappa$, which is the circular plat closure of the $m$-stranded braid $\bar\beta$ corresponding to $\beta$. Lemma \ref{th:wal} also applies to $\bar\beta$, and implies that $HF(L_{\bar{\wp}},\bar\beta(L_{\bar{\wp}})) \iso \Kh_{\mathrm{symp}}(\bar\kappa)$. The inequality \eqref{eq:wp-bar} therefore yields Corollary \ref{cor:maintwo}.

\begin{example}
Take a knot whose prime decomposition is of the form $\kappa = \#_{i=1}^s m_i \kappa_i$, where each multiplicity $m_i$ is even. One can then arrange that $\kappa$ is symmetric in the sense considered above (and the quotient is $\bar\kappa = \#_{i=1}^s (m_i/2) \kappa_i$). More interestingly, there are prime knots preserved by (many inequivalent) involutions, see for instance \cite{sakuma}.
\end{example}

\end{document}